\documentclass[11pt]{amsart}
\usepackage{amscd,amssymb,palatino}
\usepackage[utf8]{inputenc}
\usepackage[english]{babel}
\usepackage{caption}
\usepackage{etex}
\usepackage[leqno]{amsmath}
\usepackage{amssymb}
\usepackage{color}
\usepackage{shadow}
\usepackage{epsfig}
\usepackage{epic}
\usepackage{eepic}
\usepackage{graphics}
\usepackage{graphicx}
\usepackage{calc}
\usepackage{soul}

\usepackage{tikz}
\usetikzlibrary{arrows,decorations,patterns,positioning,automata,shadows,fit,shapes,calc,decorations.markings,backgrounds,scopes,decorations.text}

\usepackage{tipa}

\usepackage{calc}

\newcommand\1{\lower 9pt\hbox{\underbar{}}}
\numberwithin{equation}{section}

\newtheorem {theorem}[equation]                   {Theorem}
\newtheorem {Question}[equation]          {Question}
\newtheorem {claim}[equation]          {Claim}
\newtheorem {example}[equation]         {Example}
\newtheorem {lemma}[equation]           {Lemma}
\newtheorem {proposition}[equation]     {Proposition}
\newtheorem {conjecture}[equation]      {Conjecture}
\newtheorem {question}[equation]      {Question}
\newtheorem {corollary}[equation]       {Corollary}

\theoremstyle{definition}
\newtheorem {definition}[equation]{Definition}
\newtheorem {remark}[equation]          {Remark}

\def\R{\mathbb{R}}
\def\u{\widetilde{u}}
\def\Ca{\mathcal{C}}
\newcommand{\norm}[1]{\left\lVert#1\right\rVert}

\begin{document}
\title{The singular Weinstein conjecture}
\author{Eva Miranda}\address{ Eva Miranda,
Laboratory of Geometry and Dynamical Systems, Department of Mathematics, EPSEB, Universitat Polit\`{e}cnica de Catalunya BGSMath Barcelona Graduate School of
Mathematics in Barcelona and
\\ IMCCE, CNRS-UMR8028, Observatoire de Paris, PSL University, Sorbonne
Universit\'{e}, 77 Avenue Denfert-Rochereau,
75014 Paris, France
 }\thanks{  {Eva} Miranda  is supported by the Catalan Institution for Research and Advanced Studies via an ICREA Academia Prize 2016. Cédric Oms is supported by an AFR-Ph.D. grant of  FNR - Luxembourg National Research Fund. Eva Miranda and Cédric Oms are partially supported  by the grants reference number MTM2015-69135-P (MINECO/FEDER) and reference number {2017SGR932} (AGAUR). Eva Miranda was supported by a \emph{Chaire d'Excellence} of the \emph{Fondation Sciences Mathématiques de Paris} when this project started and this work
has been supported by a public grant overseen by the French National Research Agency (ANR) as part of the \emph{\lq\lq Investissements d'Avenir"} program (reference:
ANR-10-LABX-0098). This material is based upon work supported by the National Science Foundation under Grant No. DMS-1440140 while the authors were in residence at the Mathematical Sciences Research Institute in Berkeley, California, during the Fall 2018 semester.}
 \email{ eva.miranda@upc.edu}
 \author{ C\'edric Oms}
\address{ C\'edric Oms,
Laboratory of Geometry and Dynamical Systems, Department of Mathematics, EPSEB, Universitat Polit\`{e}cnica de Catalunya BGSMath Barcelona Graduate School of
Mathematics in Barcelona}
 \email{ cedric.oms@upc.edu}

	\begin{abstract}  In this article, we investigate Reeb dynamics on $b^m$-contact manifolds, previously introduced in \cite{MO}, which are contact away from a hypersurface $Z$ but satisfy certain transversality conditions on $Z$. The study of these contact structures is motivated by that of contact manifolds with boundary. The search  of periodic Reeb orbits on those manifolds thereby starts with a generalization of the well-known Weinstein conjecture. Contrary to the initial expectations, examples of compact $b^m$-contact manifolds without periodic Reeb orbits outside $Z$ are provided. Furthermore, we prove that in dimension $3$, there are always infinitely many periodic orbits on the critical set if it is compact.  We prove that traps for the $b^m$-Reeb flow exist in any dimension. This investigation goes hand-in-hand with the Weinstein conjecture on non-compact manifolds having compact ends of convex type. In particular, we extend Hofer's arguments to open overtwisted contact manifolds that are $\R^+$-invariant in the open ends, obtaining as a corollary the existence of periodic $b^m$-Reeb orbits away from the critical set. The study of $b^m$-Reeb dynamics is motivated by well-known problems in fluid dynamics and celestial mechanics, where those geometric structures naturally appear. In particular, we prove that the dynamics on positive energy level-sets in the restricted planar circular three body problem are described by the Reeb vector field of a $b^3$-contact form that admits an infinite number of periodic orbits at the critical set.
	\end{abstract}

	\maketitle
	\section{Introduction}
	
The existence of periodic orbits of the Hamiltonian vector field on a given level-set of a Hamiltonian function is a central question in symplectic geometry. Historically, this question is being motivated through its applications to classical mechanics and has ever since given rise to spectacular developments in the field. This problem may be treated under different additional assumptions. In the contact context the Weinstein conjecture asserts that the Reeb vector field of a closed contact manifold has at least one periodic orbit. There are several Hamiltonian and symplectic relatives of this conjecture such as the Hamiltonian Seifert conjecture or the Conley conjecture about the periodic orbits of Hamiltonian systems on a symplectic manifold. The tools of Floer theory allow to obtain refinements of these conjectures and the whole community in symplectic and contact geometry has experimented a \emph{golden age period} during the last decades.

In this alluring symplectic and contact picture, the study of singularities has been neglected. General singularities are too complicated, however in the last years, particular singularities known under the name of $b^m$-symplectic or $\log$-symplectic forms, have been widely explored by several authors {\cite{GMP, gmps, GL, log, Scott}}. The geometry of $b^m$-manifolds appear in the study of manifolds with boundary. Assuming the manifold is even-dimensional, $b^m$-symplectic structures are symplectic away from the boundary but their associated Poisson structure meets some mild degeneracy conditions on the boundary. This can be generalized to manifolds with a fixed hypersurface, called \emph{critical hypersurface}. Recently, the geometry of the odd-dimensional counterpart of those have been studied by the authors in \cite{MO}.

	Contact structures appear as regular level-sets of symplectic manifolds whenever there exists a transverse Liouville vector field. This construction is connected to the study of Hamiltonian systems. Singularities in the orbits of the Hamiltonian system (as for instance homoclinic or more generally heteroclinic orbits) hinder the dynamical description in terms of contact geometry. This yields a first motivation to analyse the singular counterpart to contact structures in order to take these situations into account.
	
	Understanding the dynamics of the Reeb vector field in contact geometry and the Hamiltonian vector field has been (and still is) a leading question in the field. The global behaviour of the Reeb vector field is fundamentally different when there are singularities in the contact form.
	
	We show that there are compact examples of $b^m$-contact manifolds in any dimension without any periodic orbits away from the critical set. Similarly, by considering the symplectization, we show that $b^m$-symplectic manifolds can have very different dynamical behaviour as examples without periodic orbits of the Hamiltonian vector field on any level-set are constructed, which is related to the Hamiltonian Seifert conjecture.
	
	In dimension $3$, the dynamics on the critical set is Hamiltonian and as a consequence we prove that for compact $Z$, there are infinitely many periodic orbits on the critical set. This is especially interesting in view of the physical applications of $b^m$-contact structures.
	
	 Even though plugs for the smooth Reeb flow cannot exist by the proof in dimension $3$  of the Weinstein conjecture (respectively by the partial positive answers of the conjecture in higher dimensions), the existence of traps (which is a weaker notion than plugs) is a subtle topic. It was proved that in dimension $3$ traps do not exist for the Reeb flow, see \cite{hofereliashberg}. In the same paper, it was conjectured that a similar result holds in higher dimensions. This was elegantly disproved: the authors of \cite{GRZ} proved that for higher dimensions, traps for the Reeb flow do exist. We prove that the dynamics of the $b^m$-Reeb flow contrasts the smooth one: traps do exist in any dimension. The construction of traps for the $b^m$-Reeb flow strongly uses the singularization technique, previously introduced in \cite{MO} to prove the existence of $b^m$-contact structures.
	
In this article we also extend
Hofer's proof of Weinstein conjecture under the assumption of existence of overtwisted disk  (see \cite{Hofer})  to non-compact manifolds whenever some additional condition is imposed on the ends. Our approach is different from previous results on the existence of periodic orbits on non-compact compact manifolds as is done for instance in \cite{tentacular}. The condition imposed on the ends is compatible with that of $b^m$-contact structures and in general for any contact manifold with boundary admitting some transversality conditions close to the boundary or, more generally, for non-compact manifolds which can be compactified via a convex hypersurface. Coming back to the case of $b^m$-contact manifolds, a necessary condition to guarantee  the existence of periodic orbits away from the critical set is that the contact structure is overtwisted away from the critical set. To overcome non-compactness, we additionally assume $\R^+$-invariance around the critical set. Under those assumptions, we prove that the $J$-holomorphic curves methods developed by Hofer in \cite{Hofer} extend to this set-up. The non-triviality of this result originates in the non-compactness of the situation. We show that the Bishop family originating from the elliptic singularity of the overtwisted disk gives rise to bubbling of a finite energy plane in the symplectization. We show that the bubbling takes place either away from the critical set or in the $\R^+$-invariant part. The invariance here plays an important role to assure uniform convergence of the pseudoholomorphic disks by translating the disks in the $\R^+$-invariant direction. Those finite energy planes, as follows from classical results due to Hofer \cite{Hofer}, show that there exists either a periodic Reeb orbit away from the critical set or infinitely many periodic Reeb orbits in the $\R^+$-invariant neighbourhood.
	
	  The above-mentioned results naturally lead to a reformulation of the Weinstein conjecture for $b^m$-contact manifolds: indeed, summarizing the above discussion, the existence of overtwisted disks away from the critical set $Z$ and the additional symmetry around it constitute a sufficient condition for the existence of periodic orbits away from $Z$. In view of the compact examples without periodic orbits together with the possible presence of singularities of the dynamics on $Z$, we are lead to rewrite Weinstein conjecture for $b^m$-contact manifold. Not only this sheds new light on the study of homoclinic and heteroclinic orbits, but eventually this can lead to far-reaching extensions of variational approaches, and thereby ultimately Floer techniques, to the singular contact and symplectic realm, and in particular to an important class of Poisson manifolds.

   {The study of the Reeb dynamics on $b^m$-contact manifolds has substantial applications to celestial mechanics: the dynamics on positive energy level-sets of the restricted planar circular three body problem are described by the flow of a $b^3$-Reeb vector field. The critical set describes the manifold at infinity. Due to non-compactness, the above explained results a priori do not apply. However, even though non-compact, we prove that there are infinitely many periodic orbits on the critical set, thereby generalizing results about periodic orbits at infinity obtained in the parabolic case in \cite{DKRS} to the hyperbolic case.

     Other applications are obtained in the terrain of fluid dynamics in view of the correspondence between $b$-contact structures and Beltrami flows on manifolds with boundary (see \cite{CMP}). In \cite{CMPP} some universality features (in the sense of \cite{tao})  are proved for regular Euler equations \cite{CMPP} which could be extended to this novel singular set-up and could be useful for the study of blow-up in finite time of Navier-Stokes equations (see for instance  \cite{tao2}).}

	\textbf{Organization of this article:} We open this article with a review on $b$-symplectic and $b$-contact geometry. We continue by motivating the study of $b^m$-contact geometry at the hand of examples coming from both celestial mechanics and fluid dynamics and outline the importance of considering singularities in those examples. In Section \ref{sec:motivatingrevisited}, taking into account the singularities present in celestial mechanics, we prove the existence of infinitely many periodic Reeb orbits at infinity in the planar restricted three body problem, the initial motivating example of this paper. In Section \ref{sec:review} we continue with a survey on well-known results in Hamiltonian and Reeb dynamics that will guide us through the results that we are going to prove for $b^m$-contact and $b^m$-symplectic manifolds. In Section \ref{sec:counterexample} we prove the existence of infinitely many periodic Reeb orbits in dimension $3$ when $Z$ is compact and give examples of compact $b^m$-contact manifolds without periodic orbits away from the critical set. As a corollary we produce examples of $b^m$-symplectic manifolds with proper Hamiltonian functions without periodic orbits on all level-sets away from the critical set. In Section \ref{sec:trap}, we prove that the singularization can be used to prove a trap construction. We will prove in Section \ref{sec:bmhofer} that in the case of an $\R^+$-invariant $b^m$-contact manifold with an overtwisted disk away from the critical set, there are always periodic Reeb orbits away from the critical set.
 The above mentioned results lead us to reformulate the Weinstein conjecture in Section \ref{sec:singularweinstein} for $b^m$-contact manifolds about the existence of singular orbits as admissible solutions and discuss open problems.

	\textbf{Acknowledgements:} We are grateful to Alain Chenciner, Jacques Féjoz, Urs Frauenfelder, Viktor Ginzburg, Andreas Knauf and Charles-Michel Marle for several key conversations during the preparation of this article. Special thanks to Francisco Presas for suggesting Hofer's approach on overtwisted contact manifolds.
We are thankful to Daniel Peralta-Salas for pointing out and providing a complete proof that in the regular Beltrami case there cannot exist invariant submanifolds of a Hamiltonian vector field  along it, which contrasts the $b$-Beltrami case. We warmly thank Amadeu Delshams and Marcel Guardia for enlightening discussions concerning the three-body problem.  We are indebted to the Fondation Sciences Mathématiques de Paris for endowing the first author with a Chaire d'Excellence in 2017-2018 when this adventure started and to the Observatoire de Paris for being a source of inspiration for many of the constructions in this article and for their hospitality during the stay of both authors during the Fall and Winter of 2017-2018. Thanks also to Robert Cardona and Arnau Planas for their help with the figures in this article.

	\section{Preliminaries: $b^m$-Symplectic and $b^m$-contact geometry}\label{b-symp}
	
	In this section, we review the basics of $b^m$-symplectic and $b^m$-contact geometry. For more details, we refer the reader to  \cite{BDMOP} and \cite{gmw1} for a review on $b^m$-symplectic structures and to \cite{MO} for an extensive study of the topology and the geometry of $b^m$-contact manifolds.
	
	Let $(M^n,Z)$ be a smooth manifold of dimension $n$ with a hypersurface $Z$. In what follows, the hypersurface $Z$ will be called \emph{critical set}. Assume that there exists a global defining function for $Z$, that is $f:M\to \R$ such that $Z=f^{-1}(0)$. The space of vector fields that are tangent to $Z$ form a Lie sub-algebra of the Lie algebra of vector fields on $M$. By the Serre--Swan theorem \cite{S}, there exists an $n$-dimensional vector bundle which sections are given by the $b$-vector fields. We denote this vector bundle by ${^b}TM$, the \emph{$b$-tangent bundle}. We denote the dual of this vector bundle by ${^b}T^*M:=({^b}TM)^*$ and call it the \emph{$b$-cotangent bundle}. A $b$-form of degree $k$ is the section of the $k$th exterior wedge product of the $b$-cotangent bundle: $\omega \in \Gamma\big(\Lambda^k({^b}T^*M)\big):= {^b}\Omega^k(M)$. By the following decomposition lemma, we can easily extend the exterior derivative.
	
	\begin{lemma}[\cite{GMP}]\label{decomposition}
		Let $\omega \in {^b}\Omega^k(M)$ be a $b$-form of degree $k$. Then $\omega$ decomposes as follows:
		$$ \omega = \frac{df}{f}\wedge \alpha +\beta, \quad \alpha \in \Omega^{k-1}(M),\ \beta \in \Omega^k(M).$$
	\end{lemma}
	The exterior derivative for $b$-forms is defined by
	$$d\omega:= \frac{df}{f}\wedge d\alpha+d\beta.$$
	Equipped with the extension of the exterior derivative, we are able to defined $b$-symplectic and $b$-contact forms.
	
	\begin{definition}
	\begin{enumerate}
		\item An even-dimensional $b$-manifold $W^{2n}$ with a $b$-form $\omega \in {^b}\Omega^2(W)$ is $b$-symplectic if $d\omega =0$ and $\omega^n \neq 0$ as a section of $\Lambda^{2n}({^b}T^*W)$.
		\item An odd-dimensional $b$-manifold $M^{2n+1}$ is $b$-contact if there exists a $b$-form $\alpha \in {^b}\Omega^1(M)$ such that $\alpha \wedge (d\alpha)^n\neq 0$. The $b$-form $\alpha$ is called $b$-contact form and the kernel $\ker \alpha \subset {^b}T^*M$ is called $b$-contact structure.
	\end{enumerate}
	\end{definition}
	
	Outside of the critical set $Z$, both definitions coincide with the usual definition of symplectic and contact manifolds respectively. The local theory of both are well-understood, as the language of differential forms for this complex gives rise to Moser's path method and Darboux local normal forms. Furthermore, $b$-symplectic, respectively $b$-contact manifolds, can be seen as particular cases of Poisson manifolds (see \cite{GMP}), respectively Jacobi manifolds (see \cite{MO}) satisfying some transversality conditions. Similarly, those transversality conditions can be relaxed by considering higher order tangencies. A similar construction to the one mentioned here leads to $b^m$-contact and $b^m$-symplectic structures, see \cite{Scott}. As one would expect, the symplectization of $b^m$-contact manifolds are $b^m$-symplectic manifolds and conversely, $b^m$-contact manifolds can be seen as hypersurface in $b^m$-symplectic manifolds that admit a transverse Liouville vector field. We emphasize that the Reeb vector field associated to a $b^m$-contact form can be singular. In fact, the vanishing of the Reeb vector field determines the local geometry of the $b^m$-contact form as is proved in the $b^m$-Darboux theorem in \cite{MO}.
	
We give two examples of $b^m$-contact manifolds that possess some interesting dynamical properties as we will see later.
	
	\begin{example}\label{example:S3}
		The $3$-dimensional sphere $S^3$ admits a $b$-contact structure. Consider the $\R^4$ with the standard $b$-symplectic structure $\omega=\frac{dx_1}{x_1}\wedge dy_1 + dx_2\wedge dy_2$ and denote by $S^3$ the unit sphere in $\R^4$. The Liouville vector field $X=\frac{1}{2}x_1\frac{\partial}{\partial x_1}+y_1\frac{\partial}{\partial y_1}+\frac{1}{2}(x_2 \frac{\partial}{\partial x_2}+y_2\frac{\partial }{\partial y_2})$ is transverse to the sphere and hence $\iota_X \omega$ defines a $b$-contact form on $S^3$. The critical set is a $2$-dimensional sphere, $S^2$, given by the intersection of the sphere with the hyperplane $x_1=0$.
	\end{example}
	
	\begin{example}\label{example:3torus}
	Consider the torus $\mathbb{T}^2$ as a $b$-manifold where the boundary component if given by two disjoint copies of $S^1$. The unit cotangent bundle $S^*\mathbb{T}^2$, diffeomorphic to the $3$-torus $\mathbb{T}^3$ is a $b$-contact manifold with $b$-contact form given by $\alpha=\sin\phi\frac{dx}{\sin(x)}+\cos \phi dy$, where $\phi$ is the coordinate on the fiber and $(x,y)$ the coordinates on $\mathbb{T}^2$.
	\end{example}
	
	To study the dynamics on the critical set in dimension $3$, we make use of the fact that the Reeb vector field restricted to the critical set is Hamiltonian.
	
		\begin{proposition}[\cite{MO}]\label{prop:ReebHamdim3}
		Let $(M,\alpha)$ be a $b$-contact manifold of dimension $3$ and let us write $\alpha=u\frac{dz}{z}+\beta$, $u\in C^\infty(M)$ and $\beta \in \Omega^1(M)$ as in Lemma \ref{decomposition}. Then the restriction on $Z$ of the $2$-form $\Theta= u d\beta+\beta\wedge du$ is symplectic and the Reeb vector field is Hamiltonian with respect to $\Theta$ with Hamiltonian function $u$, i.e. $\iota_R \Theta= du$.
	\end{proposition}
	
	We furthermore notice that $u$ is not constant on closed critical sets in dimension $3$. Indeed, if $u|_Z$ was constant, then the area form on $Z$ given by $\Theta=ud\beta$ is exact which contradicts Stokes theorem. We therefore have the following proposition.

	\begin{proposition}\label{prop:unonconstant3dim}
	Let $(M,\alpha=u\frac{dz}{z}+\beta)$ be a $3$-dimensional $b$-contact manifold with a closed critical hypersurface $Z$, where $u\in C^\infty(M)$ and $\beta \in \Omega^1(M)$ as before. Then the function $u|_Z$ is non-constant.
	\end{proposition}	
	
	In \cite{MO}, the existence of $b^m$-contact structures is solved by relating the critical set to convex hypersurfaces in contact manifolds. More precisely, the following is proved.
	
	\begin{theorem}[\cite{MO}]\label{thm:existenceb2k}
		Let $(M,\xi)$ be a contact manifold and let $Z$ be a convex hypersurface in $M$. Then $M$ admits a $b^{2k}$-contact structure for all $k$ that has $Z$ as critical set.
	\end{theorem}
	
	A similar result holds for $b^{2k+1}$-contact structures, where the critical set $Z$ is given by two diffeomorphic copies of the convex hypersurface $Z$.

\begin{theorem}[\cite{MO}]\label{thm:existenceb2k+1}
		Let $(M,\xi)$ be a contact manifold and let $Z$ be a convex hypersurface in $M$. Then $M$ admits a $b^{2k+1}$-contact structure for all $k$ that has two  connected components, both diffeomorphic to $Z$, as critical set.  Additionally, one of the hypersurfaces can be chosen to be $Z$.
	\end{theorem}

The topology of $b^m$-contact manifolds can be related to the one of contact manifolds through a desingularization technique, similar to the one introduced in \cite{gmw1}. A necessary condition for applying the desingularization is that the $b^m$-contact form is almost convex.

\begin{definition}\label{def:bcontactconvex}
		We say that a $b^m$-contact structure $(M,\ker \alpha)$ is almost convex if $\beta= \pi^*\tilde{\beta}$, where $\pi:\mathcal{N}(Z) \to Z$ is the projection from a tubular neighbourhood of $Z$ to the critical set and $\tilde{\beta}\in \Omega^1(Z)$. We will abuse notation and write $\beta \in \Omega^1(Z)$. We say that a $b^m$-contact structure is convex if $\beta \in \Omega^1(Z)$ and $u \in C^\infty(Z)$.
	\end{definition}
	
	Using almost convexity, the following is proved.
	
	\begin{theorem}[\cite{MO}]\label{thm:desingularizationb2k}
		Let $(M^{2n+1},\ker \alpha)$ a $b^{2k}$-contact structure with critical hypersurface $Z$. Assume that $\alpha$ is almost convex. Then there exists a family of contact forms $\alpha_\epsilon$ which coincides with the $b^{2k}$-contact form $\alpha$ outside of an $\epsilon$-neighbourhood of $Z$. The family of bi-vector fields $\Lambda_{\alpha_\epsilon}$  and the family of vector fields $R_{\alpha_\epsilon}$ associated to the Jacobi structure of the contact form $\alpha_\epsilon$ converges to the bivector field $\Lambda_{\alpha}$ and to the vector field $R_\alpha$ in the $C^{2k-1}$-topology as $\epsilon \to 0$.
	\end{theorem}
	
A similar results holds for $b^{2k+1}$-contact structure, where the resulting desingularization yields the so-called folded contact structures, see \cite{MO}.

\section{Motivating examples}

\subsection{Motivating examples from Celestial Mechanics}

Let us consider the restricted three body problem. This is a simplified version of the general $3$-body problem: one of the bodies has {negligible mass}.
The other two bodies called primaries move independently of it following Kepler's laws for the $2$-body problem.  In the example below we will assume these are circles and we will refer to it as  the  restricted circular three body problem. The planar version assumes that the motion occurs in a plane and we abbreviate it by RPC3BP.

\begin{figure}[hbt!]
	
	\tikzset{>=latex}
	\centering
	\definecolor{xdxdff}{rgb}{0.49019607843,0.4901960784,1.}
	\definecolor{qqqqff}{rgb}{0.1,0.,1.}
	\begin{tikzpicture}[line cap=round,line join=round,x=0.8cm,y=0.8cm]
	\clip(-6,-2.8) rectangle (7,4.5);
	\draw [fill=black,fill opacity=0.1] (-3.78,1.4) circle (1.216050846542cm);
	\draw [fill=black,fill opacity=0.05] (5.96,-1.64) circle (0.820975030076cm);
	\draw [->] (5.96,-1.64) -- (-0.3,3.92);
	\draw [->] (-3.78,1.4) -- (-0.3,3.92);
	\draw (5.96,-1.64)-- (-3.78,1.4);
	\draw [->] (0.09567095224776079,0.1903450005304732) -- (-0.3,3.92);
	\draw[color=black] (-3.7,0.8) node {$m_1 = 1 - \mu$};
	\draw[color=black] (5.9,-1.9) node {$m_2 = \mu$};
	\draw [fill=black] (-0.3,3.92) circle (1.5pt);
	\draw [fill=black] (-3.78,1.4) circle (1.5pt);
	\draw [fill=black] (5.96,-1.64) circle (1.5pt);
	\draw[color=black] (-3.78,1.7) node {$q_1$};
	\draw[color=black] (5.96,-1.3) node {$q_2$};
	\draw[color=black] (-0.16,4.29) node {$q$};
	\draw[color=black] (4.54,0.8) node {$r_2 = q - q_2$};
	\draw[color=black] (-2.54,3.2) node {$r_1 = q - q_1 $};
	\draw [fill=black] (0.09567095224776079,0.1903450005304732) circle (1.5pt);
	\draw[color=black] (-0.3,-0.3) node {$\text{center of mass}$};
	\draw[color=black] (0.2,1.8) node {$r$};
	\end{tikzpicture}
	\caption{The restricted planar three-body problem: Sun-Earth-Moon system.}
\end{figure}
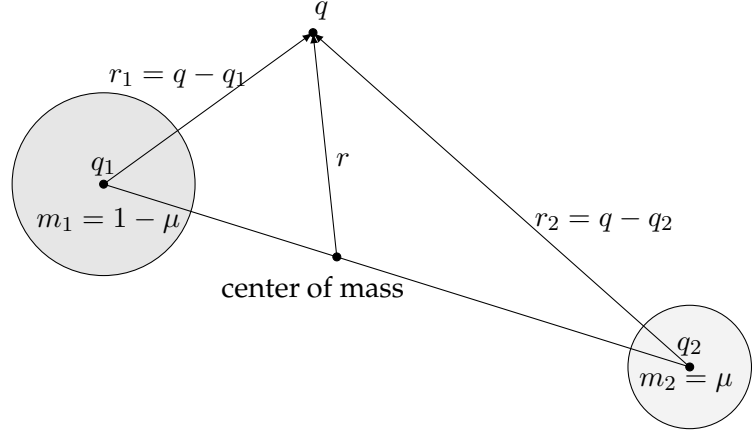

Singular symplectic structures show up naturally in examples in celestial mechanics as an offspring of regularization transformations (see for instance \cite{knauf2} and references there-in, and \cite{knauf1, gerardknauf}). In the case of the restricted $3$-body problem, these regularization transformations are due to McGehee. See also \cite{kms, dkm, BDMOP} for  examples in celestial mechanics where singular symplectic structures are analyzed in detail.

We now follow \cite{kms} and \cite{DKRS} for the description of the singular symplectic geometry of the RPC3BP.

\begin{itemize}

\item The time-dependent self-potential of the small body  is
{$U(q,t)= \frac{1-\mu}{|q-q_1(t)|} + \frac{\mu}{|q-q_2(t)|},$}
with  $q_1(t),q_2(t)$ the position of the planet with mass $1-\mu$, respectively $\mu$.

\item The Hamiltonian of the system is
\begin{equation}\label{eq:Hamiltonian3BP}
H(q,p,t)= \frac{|p|^2}{2} - U(q,t), (q,p) \in \mathbb R^2\setminus \{q_1(t),q_2(t)\} \times \mathbb R^2,
\end{equation}
where $p=\dot q$ is the momentum of the planet.

\item Passing to rotating coordinates, the positions of the two planets can be assumed to be fixed at $q_1=(\mu,0)$, respectively $q_2=(-(1-\mu),0)$. The resulting Hamiltonian is autonomous but therefore ceases to be the sum of kinetic and potential energy, see \cite{frauenfelderkoert,abrahammarsden}. It is given by
\begin{equation}\label{eq:HamiltonianPR3BP}
H(q,p)=\frac{|p|^2}{2} -\frac{1-\mu}{|q-q_1|} - \frac{\mu}{|q-q_2|}+p_1q_2-p_2q_1.
\end{equation}

\item We then consider the symplectic change of coordinates to polar coordinates given by $(q, p) \mapsto (r,\alpha, P_r, P_\alpha)$, where $q=(r \cos \alpha,r\sin \alpha)$ and $p=(P_r\cos \alpha-\frac{P_\alpha}{r}\sin \alpha,P_r\sin \alpha+\frac{P_\alpha}{r}\cos \alpha)$.

\item We then introduce the McGehee coordinates
$(x,\alpha,P_r,P_\alpha)$, where {$ r=\frac{2}{x^2}, x \in \mathbb R^{+}$}.

\item The geometric structure is a singular form given by
$-\frac{4}{x^3} dx \wedge dP_r + d\alpha \wedge d P_\alpha$
which is symplectic away from the line at infinity and extends to a singular symplectic form (technically called $b^3$-symplectic structure)  on $\mathbb{R}^+ \times \mathbb S^1 \times \mathbb R^2$.

\end{itemize}

As it is customary in the classical theory of symplectic and contact geometry, the restriction of the symplectic form on regular level-sets of $H$ induces a contact structure whenever there exists a \emph{Liouville vector field} that is transverse to it.
In this new picture this contact structure may have singularities.

In \cite{albers} the authors apply results from contact topology to prove existence of periodic orbits in the RPC3BP using a regularization introduced by Moser. The study of the topology of the problem strongly depends on the geography of the Lagrangian points, which are the critical points of the Hamiltonian, depicted below in Figure \ref{fig:lagrangepoints}.

\begin{center}
\begin{figure}[hbt!]
\begin{tikzpicture}[scale=0.8]

   \draw (0,0,0) circle(2.7);
   \draw[fill=black!30] (0,0,0) circle(0.4);
   \draw[fill=black!30] (2.7,0,0) circle (0.2);
   \draw (2.7,0,0) circle (0.5);
   \draw[fill=black!30] (3,0.38,0) circle (0.1);

   \draw (-2.7,0,0) node[circle,fill,scale=0.4]{};
   \draw (1.1,2.46,0) node[circle,fill,scale=0.4]{};
   \draw (1.1,-2.46,0) node[circle,fill,scale=0.4]{};
   \draw (1.8,0,0) node[circle,fill,scale=0.4]{};
   \draw (3.6,0,0) node[circle,fill,scale=0.4]{};

   \draw (-3.1,0,0) node{$L_3$};
   \draw (1.2,2.7,0) node{$L_4$};
   \draw (1.2,-2.8,0) node{$L_5$};
   \draw (1.8,-0.4,0) node{$L_1$};
   \draw (3.6,-0.4,0) node{$L_2$};

   \end{tikzpicture}
   \caption{Lagrangian points}
   \label{fig:lagrangepoints}
   \end{figure}
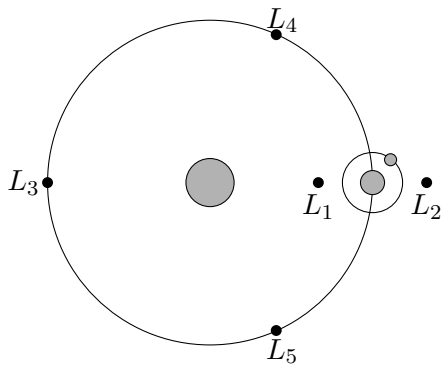
\end{center}

Following \cite{albers}, for low energy levels $c \in \mathbb{R}$ of the Hamiltonian, the level-set $\Sigma_c=H^{-1}(c)$ has 3 connected components. The component where the position of the satellite is bounded in position around the earth (respectiely moon) is denoted by $\Sigma_c^E$ (respectively $\Sigma_c^M$). The third component corresponds to the case in which the satellite is far away in position from the primaries.

  The first Lagrangian point $L_1$ lies on the axis between earth and moon and it is a critical point of the
energy.  As discussed in \cite{albers} if $c=H(L_1)+\epsilon$ (for $\epsilon$ sufficiently small) the satellite can cross from the region around the earth to the region around the
moon and in this case there are two connected components, one bounded (which we denote by $\Sigma_c^{E,M}$) and one unbounded.

 To deal with the singularities of the Kepler problem, Moser \cite{moser} introduced a regularization procedure for $2$-body collisions.  This can be applied to the planar circular restricted $3$-body problem. Using Moser's regularization the components $\Sigma_c^E$ and $\Sigma^M_c$ can be compactified to $\overline{\Sigma}_c^E$ and $\overline{\Sigma}_c^M$  which are diffeomorphic to the real projective space $\mathbb{R}P^3$ and {$\overline{\Sigma}_c^{E,M}$} is
diffeomorphic to {$\mathbb{R}P^3 \# \mathbb{R}P^3$.

\begin{theorem}[\cite{albers}]\label{thm:albers}
For $c<H(L_1)$ the contact structures
$\big(\overline{\Sigma}_c^E,\ker \alpha\big)$ and
$\big(\overline{\Sigma}_c^M,\ker \alpha\big)$ coincide with the
{ standard contact structure on $\mathbb{R}P^3$ which is tight} and for $c \in (H(L_1),H(L_1)+\epsilon)$ the
contact structure $\big(\overline{\Sigma}_c^{E,M},\ker \alpha\big)$
coincides with the tight $\mathbb{R}P^3\#\mathbb{R}P^3$.
\end{theorem}

A relevant outcome of this theorem is the application of contact topology results to the actual problems in celestial mechanics. More concretely, the celebrated Weinstein conjecture (proved in dimension $3$ by Taubes \cite{taubes}) claims that {the Reeb vector field of a contact compact manifold admits at least one periodic orbit}.

The combination of Weinstein conjecture with Theorem \ref{thm:albers} yields,

\begin{theorem}[\cite{albers}] For any value {$c<H(L_{1})$}, the regularized planar circular restricted
three body problem has a closed orbit with energy $c$.
\end{theorem}

 The results of \cite{albers} identify the contact topology on the level-sets of $H$ and prove existence of periodic orbits but do not localize these orbits with respect to the line at infinity. This is mainly due to the fact that the Moser's regularization and compactification gets rid of this singular set. However, this is not the case of McGehee regularization where the line at infinity is identified with the singular set of the $b^3$-symplectic structure. For astrodynamical purposes  it is convenient to be able to understand periodic orbits getting close to the line at infinity and periodic orbits at the line at infinity. Considering the new $b^3$-symplectic model described above, one can raise the following question:

\begin{Question}
  Are the dynamics on the level-sets of the Hamiltonian described by the flow of a Reeb vector field associated to a singular contact form? Can we use new methods in singular contact geometry to localize these periodic orbits with respect to the line at infinity?
\end{Question}

More generally, to understand the dynamical behavior of those manifolds, it is natural to ask about the generalization of Weinstein conjecture in this setting as this is, by the above considerations, strongly related to the existence of periodic orbits in the RPC3BP.

\begin{Question}
{Does the Weinstein conjecture hold in this singular set-up?}
\end{Question}

\subsection{Motivating examples from Fluid Dynamics}\label{subsec:fluiddynamics}

Euler equations model the dynamics of an inviscid and incompressible fluid flow. Their viscid counterpart yield the Navier-Stokes equations.

Euler equations can be generalized from the Euclidean to the general Riemannian case as follows: On a Riemannian $3$-manifold $(M^3,g)$ they can be described by
\begin{align*}
{ \frac{\partial X}{\partial t} + \nabla_X X} &= - {\nabla P} \\
	{ \operatorname{div}X}&={0}
\end{align*}
where $X$ is the velocity, $\nabla$ the Riemannian gradient and $P$ the pressure. The Bernoulli function is given   by  $B= P + \frac{1}{2}g(X,X)$. We can take advantage of the metric $g$ to identify several classical concepts using Riemannian duality as follows: The vorticity vector $\omega$
is defined as $$  \iota_\omega \mu= d\alpha $$ where $\alpha= \iota_X g$ and $\mu$ the Riemannian volume.

 Even though, the classical work in the subject uses the language of vector calculus. In this article we adopt the language of forms. For more information about the subject please consult \cite{peraltasalas2}.
%
%

When the flow does not depend on time we obtain the so-called \emph{stationary solutions}. 
 In terms of $\alpha = \iota_X g$,  stationary Euler equations can be written as,
$$
  \left\{
    \begin{array}{l}
      \iota_X d\alpha = -dB\\
     d\iota_X\mu = 0.
    \end{array}
  \right.
$$

An important class of stationary solutions are given by Beltrami fields which satisfy
$$ \operatorname{curl}X=fX, \text{ with } f \in C^\infty (M). $$
When $f\neq0$, we call those vector field \emph{rotational}.
Among the examples of Beltrami fields we find Hopf fields and ABC flows.

If $X$ is non-vanishing rotational Beltrami then {$\alpha= \iota_X g$} is a contact structure. In order to prove this note that
the Beltrami equation in the language of forms described above can be written as $d\alpha = f \iota_X \mu $.
Since $f$ is strictly positive and $X$ is not vanishing  we obtain: {$ \alpha \wedge d\alpha= f \alpha \wedge \iota_X \mu > 0.$} Thus proving that $\alpha$ is  a contact structure.

Further, the vector field
 $X$ satisfies {$\iota_X (d\alpha)= \iota_X \iota_X \mu=0$} so $X \in \ker d\alpha$. This implies that it is a reparametrization of the Reeb vector field by the function $\alpha(X)= g(X,X)$.

This proves one of the implications of  the theorem below proved  Etnyre and Ghrist \cite{EG}:

\begin{theorem}\label{correspondence}
Any nonsingular rotational Beltrami field is a reparametrization of a Reeb vector field for some contact form and conversely any reparametrization of a Reeb vector field of a contact structure is a nonsingular rotational Beltrami field for some metric and volume form.
\end{theorem}

The geometrical approach to hydrodynamics probably started with Arnold \cite{Ar65} (see also the joint book with Khesin \cite{{AK}}).
In \cite{EP1, EP2, enciso} the authors exploit the geometrical flavour of stationary solutions to the Euler equations to study knots, links and vortex tubes in this context solving in particular a conjecture of Lord Kelvin in \cite{EP2}.

In \cite{CMP}  contact manifolds with boundary having a singular contact structure on the boundary of \emph{$b$-type} are identified with contact manifolds with boundary where the boundary is pushed to infinity (or manifolds with cylindrical ends). Using this identification in \cite{CMP} it is proved that the correspondence contact/Beltrami can be  extended   to the singular set up thus extending the previous geometrical picture on Beltrami fields to $3$-dimensional manifolds with boundary.

The  correspondence between non-singular Beltrami fields and regular contact structures holds in any odd dimension (see \cite{CMPP} for a proof) and similarly higher dimensional singular contact structures naturally arise associated to Euler's equations. In \cite{CMPP} the authors prove universality properties for Euler's equations via a Reeb embedding theorem which we prove thanks to refining the  $h$-principle techniques in \cite{bem}. Extending those techniques to the singular realm would allow to understand if such universality properties still hold on manifolds with cylindrical ends.

	 \section{Motivating examples revisited}\label{sec:motivatingrevisited}
	
	 In this section, we revisit the motivating examples and use the view-point of $b$-contact geometry as explained in the preliminaries to deepen the understanding of the dynamics in both the RPC3BP and Beltrami vector fields. Those examples make apparent the importance of an exhaustive knowledge of the Reeb dynamics for $b^m$-contact manifolds.
	
	 \subsection{Infinitely many periodic orbits on the manifold at infinity in the RPC3BP}\label{sec:cr3bp}

We will apply the results on the dynamics of the $b^m$-Reeb flow already introduced in Proposition \ref{prop:ReebHamdim3} and \ref{prop:unonconstant3dim} to prove that there exists infinitely many periodic Reeb orbits on the manifold at infinity.

It was proved in \cite{dkm} (see also \cite{BDMOP}) that the underlying geometric structure in the restricted three body problem after the McGehee change of coordinates is a $b^3$-symplectic structure. In this section, we focus on level-sets of the Hamiltonian. In particular, we study the motion of the satellite at infinity and therefore assume that $H=c>0$. We will see that in this case the dynamics are described by the Reeb flow of a $b^3$-contact form and that there are infinitely many periodic Reeb orbits on the manifold at infinity.

Our approach contrasts the one given in \cite{albers}, where the authors use Moser's regularization (see also \cite{knauf2}) to show that the level-sets can be regularized to compact contact manifolds. In our case, the existence of periodic orbits on the critical set follows from the observation that the Reeb vector field in dimension $3$ is a Hamiltonian vector field on the critical set, see Proposition \ref{prop:ReebHamdim3}. First, let us pass to polar coordinates $(r,\alpha,P_r,P_\alpha)$ through a symplectic change of coordinates.

This change of coordinates is given by $q=(r\cos \alpha, r\sin \alpha)$ and $p=(P_r\cos \alpha-\frac{P_\alpha}{r}\sin \alpha,P_r\sin \alpha+\frac{P_\alpha}{r}\cos \alpha)$ and the symplectic form is given by $\omega=\sum_{i=1}^2 dq_i\wedge dp_i= dr\wedge dP_r +d\alpha\wedge dP_\alpha$ and we then perform the McGehee change of coordinates, given by
\begin{equation}\label{eq:McGehee}
    r=\frac{2}{x^2}.
\end{equation}

The symplectic form then gives rise to a $b^3$-symplectic form which can be written as:

\begin{equation}\label{eq:b3symplecticform}
-\frac{dx}{x^3}\wedge dy +d\alpha \wedge dG.
\end{equation}

We now look at the level-sets of $H$ under those coordinate changes. In contrast to \cite{albers}, as the McGehee change of coordinate exchanges infinity with the origin, we consider the level-sets $\Sigma_c$ such that $\pi(\Sigma_c)$ is unbounded: indeed, we will only consider $c>0$. Furthermore, we do not consider the Liouville vector field in the position coordinates, that is $X=(q-q_M)\frac{\partial }{\partial q}$, but the one given by momenta. The reason for this is that $X$ is not a $b^3$-vector field and therefore the contraction $\iota_X \omega$ does not give rise to a $b^3$-form.

We first check that the Liouville vector field in  momenta is transverse to the positive energy level-sets before doing the McGehee change of coordinates.

\begin{lemma}
The vector field $Y=p\frac{\partial}{\partial p}$ is a Liouville vector field and is transverse to $\Sigma_c$ for $c>0$.
\end{lemma}

\begin{proof}
The vector field $Y$ is a Liouville vector field as $\mathcal{L}_Y(\sum_{i=1}^2 dp_i \wedge dq_i)=\omega$ and is transverse to $\Sigma_c$ for $c>0$. Indeed

$$ Y(H)=|p|^2+p_1q_2-p_2q_1=\frac{|p|^2}{2}+\frac{1-\mu}{|q-q_E|}+\frac{\mu}{|q-q_M|}+H(q,p).
$$
Hence $Y(H)|_{H=c}=\frac{|p|^2}{2}+\frac{1-\mu}{|q-E|}+\frac{\mu}{|q-M|}+c$ which is a sum of positive terms when $c>0$.
\end{proof}

We now prove that the vector field $Y$ is also transverse to the level-sets of the Hamiltonian at infinity. The strategy of this is to do the McGehee change of coordinates and check if the vector field is still transverse to the level-set of the Hamiltonian.

\begin{theorem}\label{thm:bcontact3bp}
After the McGehee change, the Liouville vector field $Y=p\frac{\partial}{\partial p}$ is a $b^3$-vector field that is everywhere transverse to $\Sigma_c$ for $c>0$ and the level-sets $(\Sigma_c,\iota_Y \omega)$ for $c>0$ are $b^3$-contact manifolds. Topologically, the critical set is a cylinder and the Reeb vector field admits infinitely many non-trivial periodic orbits on the critical set.
\end{theorem}

\begin{proof}
Let us compute the Hamiltonian given by Equation \ref{eq:HamiltonianPR3BP} first in polar coordinates and then perform the McGehee change of coordinate.
The polar coordinates are defined by the position $q=(r\cos\alpha,r\sin\alpha)$, $(r,\theta)\in \R^+\times S^1$, and the momenta $p=(P_r\cos \alpha-\frac{P_\alpha}{r}\sin\alpha,P_r\sin \alpha+\frac{P_\alpha}{r}\cos \alpha)$, $(P_r,P_\alpha)\in \R^2$. Under this coordinate change, the resulting Hamiltonian is given by the following expression:
\begin{align*}
    &H(r,\alpha,P_r,P_\alpha)\\
    =&\frac{1}{2}(P_r^2-(\frac{P_\alpha}{r})^2)-\frac{1-\mu}{r^2-2\mu r \cos \alpha+\mu^2}-\frac{\mu}{r^2-2(1-\mu)r\cos \alpha+(1-\mu)^2}-P_\alpha.
\end{align*}

The coordinate change is symplectic and therefore the symplectic form is given by $dr\wedge d\alpha +dP_r\wedge dP_\alpha$ and the Liouville vector field writes down $Y=P_r\frac{\partial}{\partial P_r}+P_\alpha \frac{\partial}{\partial P_\alpha}$.

After the McGehee change of coordinates $r=\frac{2}{x^2}$, the Hamiltonian is given by
\begin{align*}
&H(x,\alpha,P_r,P_\alpha)\\
=&\frac{1}{2}(P_r^2-\frac{1}{4}x^4P_\alpha^2)-x^4\frac{1-\mu}{4-4\mu x^2\cos\alpha +\mu^2x^4}- x^4\frac{\mu}{4-4x^2(1-\mu)\cos \alpha+(1-\mu)^2 x^4}- P_\alpha.
\end{align*}

The Liouville vector field does not change under the McGehee change of coordinates, but instead of a symplectic form, the underlying geometric structure is a $b^3$-symplectic structure with critical set given by $\{x=0\}$ given by $\omega=-4\frac{dx}{x^3}\wedge dPr +d\alpha \wedge dP_\alpha$. We already checked that the Liouville vector field is everywhere transverse to the level-set of $H$ and we now check that it is also transverse at the critical set.

On the critical set, the Hamiltonian is given by $H=\frac{1}{2}P_r^2-P_\alpha$, so that $Y(H)=P_r^2- P_\alpha$. On the level-set $H=c>0$, we obtain $Y(H)=\frac{1}{2}P_r^2+c>0$. Hence it is transverse to the critical set as well, and therefore the induced $b^3$-contact form on the critical set is given by $\alpha=(P_r\frac{dx}{x^3}+P_\alpha d\alpha)|_{H=c}$.

The critical set of the $b^3$-contact manifold is given by $Z=\{(x,\alpha,P_r,P_\alpha)|x=0,\frac{1}{2}P_r^2-P_\alpha=c \}$. Topologically, the critical set of the $b^3$-contact manifold is given by $Z=\{(x,\alpha,P_r,P_\alpha)|x=0,\frac{1}{2}P_r^2-P_\alpha\}$. Topologically, the critical set is a cylinder, as solutions for $\frac{1}{2}P_r^2-P_\alpha=c$ are given by $P_\alpha=\frac{1}{2}P_r^2-c:=P_\alpha(P_r)$. The cylinder is described by $Z=\{0,\alpha,P_r,P_\alpha(P_r)\}$ and hence non-compact.

According to the decomposition lemma, the $b^3$-contact form decomposes as $\alpha=u\frac{dx}{x^3}+\beta$ and by Proposition \ref{prop:ReebHamdim3}, the Reeb vector field on the critical set is Hamiltonian for the Hamiltonian function $u$. The Hamiltonian function here is given by $P_r$. As the Hamiltonian vector field is contained in the level-set of the Hamiltonian, we obtain that both cylinder are foliated by non-trivial periodic orbits away from $P_r=0$.
\end{proof}

A reformulation of Theorem \ref{thm:bcontact3bp} from a view-point of dynamical system is the following:

\begin{corollary}
After the McGehee change in the RPC3BP, there are infinitely many non-trivial periodic orbits at the manifold at infinity for energy values of $H=c>0$ (that is hyperbolic motion).
\end{corollary}

Periodic orbits at infinity have been studied in the past to successfully show oscillatory motions in the RPC3BP, see \cite{gms}, as well as to show global instability, see \cite{DKRS}. The result presented here in fact generalizes the result on the existence of periodic orbits in \cite{DKRS}, where the authors consider parabolic motions. As we consider positive energy level-sets, the motion considered here is classically known as hyperbolic motion. The authors believe that the introduced techniques in this paper do not only provide understanding of the dynamics at the manifold at infinity, as is presented in the last result, but also away from the critical set by applying perturbation methods (continuation methods, KAM theory,...) to the set-up provided in this paper. This will be tackled in an upcoming paper.

\subsection{$b$-Reeb dynamics and Beltrami vector fields}

In this section we go back to the example of $b$-contact in the context of Euler flows and $b$-Beltrami fields. We will prove  the following proposition which was explained to us by Daniel Peralta-Salas \cite{peraltasalas} but which we include here for the sake of completeness. The proof follows  the lines as the proof of Proposition 27 in \cite{peraltasalas2}.

\begin{proposition}
Consider a closed surface $\Sigma \subset M$. Assume that $\Sigma$ is invariant by a smooth vector field $X$ . Then if $X$ is a rotational Beltrami, its restriction $X|_\Sigma$ cannot be Hamiltonian.
\end{proposition}

\begin{proof}
Let us denote  by $j:\Sigma \to M$ the inclusion and let us assume the opposite, that is that $j^*X$ is Hamiltonian, i.e., $j^*X=X_H$ for $H\in C^\infty(\Sigma)$. Then by compactness, $H$ attains its extrema on $\Sigma$ and furthermore the zeros of $j^*X$ are non-degenerate and therefore isolated.

To see this, let us first denote as in Subsection \ref{subsec:fluiddynamics} $\alpha=\iota_X g$. Note that $j^*\alpha$ is closed because $\iota_X d\alpha=0$ and therefore locally exact, hence there exists a function $F\in C^\infty(\Sigma)$ such that $j^*\alpha=dF$. As $\alpha=g(X,\cdot)$, this is saying that the vector field $j^*X$ is the gradient of $F$. As $X$ is divergence free, $F$ is in fact harmonic and therefore the critical points of $F$ are isolated (see for instance \cite{morse}).

There exists hence contractible periodic orbits of $X$ around the extrema of $H$ on $\Sigma$. Let us denote by  $\gamma$ one of these orbits and the disk supporting $\gamma$ by $D$. Let $\sigma$ be an area form on the disk. By Stokes theorem and using the definition of Beltrami vector fields (that is $\operatorname{curl} X=fX$),
$$0<\int_\gamma X ds =\int_D \operatorname{curl} X \cdot N d\sigma =\int_D fX \cdot N d\sigma=0$$
because $X$ is tangent to $\Sigma$. This is a contradiction and hence $j^*X$  cannot be Hamiltonian.
\end{proof}

This result comes as a surprise in view of the following:
As an outcome of Proposition \ref{prop:ReebHamdim3} in the  $3$-dimensional $b$-contact case the Reeb vector field is tangent to the critical set $Z$ and Hamiltonian along $Z$. Now consider the $b$-Beltrami case presented in Subsection \ref{subsec:fluiddynamics}. The critical set of the associated $b$-contact structure is an invariant manifold which is Hamiltonian along $Z$ thus proving that new interesting dynamics emerge from the existence of the critical locus.

	\section{Review of classical results in Hamiltonian and Reeb dynamics}\label{sec:review}

In this section, we outline a survey on the dynamical results in contact and symplectic geometry focusing on the Weinstein conjecture. We will review important breakthroughs concerning the Weinstein conjecture, the Hamiltonian Seifert conjecture, as well as an exposition of traps and plugs.

\subsection{The Weinstein conjecture}\label{subsec:reviewweinstein}

	 The well-known Weinstein conjecture asserts the following:
	
	\begin{conjecture}[Weinstein conjecture]
		Let $(M,\alpha)$ be a closed contact manifold. Then there exists at least one periodic Reeb orbit.
	\end{conjecture}
	
	The Weinstein conjecture is still open in full generality but has been proved in several cases, the most striking positive answers being Hofer's proof in the presence of overtwisted disks, see \cite{Hofer,abbashofer} and Taubes proof in dimension $3$, see \cite{taubes}. We are going to outline Hofer's proof for overtwisted contact manifolds.
	
	\begin{definition}\label{def:OTdisk}
A $3$-dimensional contact manifold $(M,\xi=\ker\alpha)$ is called overtwisted if there exists a an embedded disk $D^2$ such that the boundary of $T\partial D \subset \xi|_{\partial D}$ and $TD \cap \xi$ defines a $1$-dimensional foliation except on a unique elliptic\footnote{For a precise definition of elliptic singular point, we refer to Page $55$ of \cite{abbashofer}.} singular point $e \in \text{int} D$ with $T_eD=\xi_p$. The disk $D$ is called overtwisted disk and we will denote it by $D_{OT}$. The point $e$ is called the elliptic singularity.
	\end{definition}
	
	A contact manifold that is not overtwisted is called \emph{tight}.
	
\begin{figure}[!ht]\label{fig:overtwisteddisk}
\centering	
\begin{tikzpicture}
    \draw (0,0) circle (2cm);
    \draw (0,0) ..controls +(0.3,0.5) and +(-0.3,0.3).. (1,0);
    \draw (1,0) ..controls +(0.3,-0.3) and +(0,-0.5).. (1.9,0);
    \draw[rotate=22] (0,0) ..controls +(0.3,0.5) and +(-0.3,0.3).. (1,0);
    \draw[rotate=22] (1,0) ..controls +(0.3,-0.3) and +(0,-0.5).. (1.9,0);
    \draw[rotate=45] (0,0) ..controls +(0.3,0.5) and +(-0.3,0.3).. (1,0);
    \draw[rotate=45] (1,0) ..controls +(0.3,-0.3) and +(0,-0.5).. (1.9,0);
    \draw[rotate=69] (0,0) ..controls +(0.3,0.5) and +(-0.3,0.3).. (1,0);
    \draw[rotate=69] (1,0) ..controls +(0.3,-0.3) and +(0,-0.5).. (1.9,0);
     \draw[rotate=90] (0,0) ..controls +(0.3,0.5) and +(-0.3,0.3).. (1,0);
    \draw[rotate=90] (1,0) ..controls +(0.3,-0.3) and +(0,-0.5).. (1.9,0);
    \draw[rotate=112] (0,0) ..controls +(0.3,0.5) and +(-0.3,0.3).. (1,0);
    \draw[rotate=112] (1,0) ..controls +(0.3,-0.3) and +(0,-0.5).. (1.9,0);
    \draw[rotate=135] (0,0) ..controls +(0.3,0.5) and +(-0.3,0.3).. (1,0);
    \draw[rotate=135] (1,0) ..controls +(0.3,-0.3) and +(0,-0.5).. (1.9,0);
    \draw[rotate=157] (0,0) ..controls +(0.3,0.5) and +(-0.3,0.3).. (1,0);
    \draw[rotate=157] (1,0) ..controls +(0.3,-0.3) and +(0,-0.5).. (1.9,0);
    \draw[rotate=180] (0,0) ..controls +(0.3,0.5) and +(-0.3,0.3).. (1,0);
    \draw[rotate=180] (1,0) ..controls +(0.3,-0.3) and +(0,-0.5).. (1.9,0);
    \draw[rotate=202] (0,0) ..controls +(0.3,0.5) and +(-0.3,0.3).. (1,0);
    \draw[rotate=202] (1,0) ..controls +(0.3,-0.3) and +(0,-0.5).. (1.9,0);
    \draw[rotate=225] (0,0) ..controls +(0.3,0.5) and +(-0.3,0.3).. (1,0);
    \draw[rotate=225] (1,0) ..controls +(0.3,-0.3) and +(0,-0.5).. (1.9,0);
    \draw[rotate=247] (0,0) ..controls +(0.3,0.5) and +(-0.3,0.3).. (1,0);
    \draw[rotate=247] (1,0) ..controls +(0.3,-0.3) and +(0,-0.5).. (1.9,0);
    \draw[rotate=270] (0,0) ..controls +(0.3,0.5) and +(-0.3,0.3).. (1,0);
    \draw[rotate=270] (1,0) ..controls +(0.3,-0.3) and +(0,-0.5).. (1.9,0);
    \draw[rotate=292] (0,0) ..controls +(0.3,0.5) and +(-0.3,0.3).. (1,0);
    \draw[rotate=292] (1,0) ..controls +(0.3,-0.3) and +(0,-0.5).. (1.9,0);
    \draw[rotate=315] (0,0) ..controls +(0.3,0.5) and +(-0.3,0.3).. (1,0);
    \draw[rotate=315] (1,0) ..controls +(0.3,-0.3) and +(0,-0.5).. (1.9,0);
    \draw[rotate=337] (0,0) ..controls +(0.3,0.5) and +(-0.3,0.3).. (1,0);
    \draw[rotate=337] (1,0) ..controls +(0.3,-0.3) and +(0,-0.5).. (1.9,0);
\end{tikzpicture}
\caption{An overtwisted disk}
\end{figure}
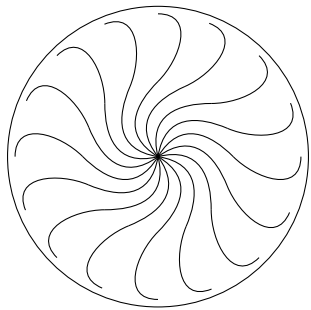
	
	The proof is based on $J$-holomorphic curve techniques applied to the symplectization of the contact manifold $(M,\alpha)$. The almost-complex structure $J$ in the symplectization is compatible with the contact form.
	
	 More precisely, the almost-complex structure considered in the symplectization of the contact manifold $(M,\alpha)$ is constructed as follows. First fix a complex structure $J_\xi$ on the plane-field $\xi=\ker \alpha$ that is compatible with $\alpha$, i.e. $d\alpha(J_\xi\cdot,J_\xi\cdot)=d\alpha(\cdot,\cdot)$ and $d\alpha(\cdot, J_\xi\cdot)>0$. We then extend the complex structure $J_\xi$ on $\xi$ to an almost complex structure $J$ on $M\times \R$, compatible with $\omega=d(e^t\alpha)$ in the following way:
\begin{itemize}
    \item $J|_\xi=J_\xi$,
    \item $\omega(J\cdot,J\cdot)=\omega(\cdot,\cdot)$,
   \item $\omega(\cdot,J\cdot)>0$,
   \item $J(\frac{\partial}{\partial t})= R_\alpha$.
\end{itemize}

A $J$-holomorphic curve is a map from a Riemann surface punctured in a finite set $\Gamma$ to the symplectization $\u:(\Sigma \setminus \Gamma,j) \to (\R \times M,J)$ satisfying  the non-linear Cauchy--Riemann equation $d\u\circ j=J \circ d\u$. Given a $J$-holomorphic curve, the energy is defined by
$$E(\u)= \sup_{\phi \in \Ca}\int_{\Sigma \setminus \Gamma}\u^*d(\phi \alpha)$$
where $\Ca$ is the set of all smooth maps $\phi:\R \to [0,1]$ satisfying that $\phi'\geq 0$. In what follows, we will always denote $ \widetilde{u}=(a,u)$, where $a:\Sigma\setminus \Gamma \to \R$ and $u:\Sigma \setminus \Gamma \to \R$. The \emph{horizontal energy}, also called the $d\alpha$-energy, is defined to be
$$E^h(u)=\int_{\Sigma \setminus \Gamma} u^* d\alpha. $$

It is clear from the definitions that $E^h(u)\leq E(\u)$.
Hofer proved the following result, reducing the quest for periodic Reeb orbits to the existence of non-constant finite energy planes. Those finite energy planes arise from a careful bubbling-off analysis à la Uhlenbeck--Sachs.

\begin{theorem}\label{thm:finiteenergy}
    Let $\u:\mathbb{C} \to \R\times M$ be a non-constant $J$-holomorphic plane such that $E(\u)<\infty$. Then there exists at least one periodic Reeb orbit in $M$.
\end{theorem}
The $J$-holomorphic curve $\u$ as in Theorem \ref{thm:finiteenergy} is called \emph{finite energy plane}.

 Out of the data of the overtwisted disk, Hofer proved the existence of a family of $J$-holomorphic curves emanating from the elliptic point $e$ of the overtwisted disk $D_{OT}$ that satisfy some additional properties which are key to study this family. We denote by $D_{OT}^*=D_{OT}\setminus\{e\}$.

\begin{theorem}\label{thm:bishop}
    Let $D$ be the $2$-disk. There is a continuous map
    $$\Psi: D \times [0,\epsilon) \to \R\times M$$
    such that for each $\u_t(\cdot) = \Psi(\cdot, t)$
    \begin{enumerate}
        \item $\u_t: D \to M\times \R$ is $J$-holomorphic,
        \item $\u_t(\partial D) \subset D_{\text{OT}}^* \subset \{0\}\times M$ for $t\in (0,\epsilon)$,
        \item $\u_t|_{\partial D}: \partial D \to D_\text{OT}^*$ has winding number $1$ for $t\in (0,\epsilon)$,
        \item $\Psi|_{D\times (0,\epsilon)}$ is a smooth map,
        \item $\Psi(z,0)=e$ for all $z\in D$,
        \item $Ind(\u_t)=2$.
    \end{enumerate}
\end{theorem}

The family $\{\u_t\}_{t\in[0,\epsilon[}$ is called \emph{Bishop family}. It is then studied whether or not the family can be extended. First, as an application of the maximum principle, the Bishop family restricted to the boundary is necessarily transverse to the characteristic foliation on the overtwisted disk and foliate a neighbourhood of the elliptic singularity $e$ by circles.

\begin{lemma}\label{lem:transverse}
The Bishop family is transverse to transverse to the characteristic foliation of the overtwisted disk $\xi|_{D_{OT}^*} \cap TD_{OT}^*$.
\end{lemma}

It is shown that if the gradient of $\u_t$ is uniformly bounded in the interval $[0,T]$, then the family $\{u_t\}_{t\in[0,\epsilon[}$ can be maximally extended. However, this results leads to a contradiction with the transversality in Lemma \ref{lem:transverse}. Hence the norm of the gradient blows-up. There are basically two different possibilities for the gradient to blow up: it can blow up at the boundary of the $J$-holomorphic disk or in the interior.  A careful analysis then shows that in the case where the gradient blows up on the boundary, so called disk bubbling happens, which {once more} contradicts transversality of the Bishop family with the characteristic foliation. The only possibility is that the Bishop family blows up in the interior. The blow-up of the norm of the gradient in the interior of the disk is giving rise to bubbling phenomena. A carefully chosen reparametrization of the bubble converges uniformly to a non-constant finite energy plane. Hence by Theorem \ref{thm:finiteenergy} there exists a periodic Reeb orbit.

	\subsection{Hamiltonian Seifert conjecture}
	Contact manifolds can be seen as a particular case of a level-set of a Hamiltonian $H$ in symplectic manifolds, where the Reeb flow is a reparametrization of the Hamiltonian flow.
	
	In the set-up of Hamiltonian dynamics, periodic orbits are in a one-to-one correspondence with the critical points of the action functional $\mathcal{A}_H$. The action of a contractible loop on a symplectic manifold $(W,\omega)$ is given by $$\mathcal{A}_H(\gamma)= \int_{D^2} u^*\omega + \int_{S^1} H(\gamma(t))dt,$$ where $u:D^2 \to W$ is such that $u(\partial D^2)=\gamma$. Here $\gamma$ is assumed to be periodic.
	
	Powerful variational methods arise from the least action principle. For instance is known that ``almost all" level-sets contain periodic orbits of the Hamiltonian flow for a large class of symplectic manifolds. More precisely, let us mention the following almost-existence theorem:

	\begin{theorem}[\cite{hoferzehnder}]\label{almostexistence}
		Let $(M,\omega)$ be a symplectic manifold of finite Hofer--Zehnder capacity. Then for all $H: M \to \R$ such that $\{H \leq a\}$ is compact, almost all level-sets contain periodic orbits.
	\end{theorem}

 A value $a$ of a Hamiltonian $H$ is called \emph{aperiodic} if the level $\{H = a\}$ carries no periodic orbits and we denote by $\mathcal{AP}_H$ the set of aperiodic orbits. Theorem \ref{almostexistence} can be restated that $\mathcal{AP}_H$ is of measure zero for many symplectic manifolds.

  The Hamiltonian Seifert conjecture states that $\mathcal{AP}_H$ for $H$ a proper, smooth function in $(\R^{2n},\omega_{\text{st}})$ is empty. The conjecture is known to be false:


\begin{theorem}[\cite{Ginzburg2}]\label{thmginzburg}
  Let $2n \geq 6$. There exists a smooth function $H:\R^{2n} \to \R$ such that the flow of $X_H$ does not have any closed orbits on the level set $\{H=1 \}$.
\end{theorem}

The theorem was independently proved by Herman \cite{Herman} for {$C^2$-Hamiltonian functions}. In dimension $4$, a $C^2$-counterexample is proved in \cite{GinzburgGurel}.

	We conclude from the last two results that for many manifolds, $\mathcal{AP}_H$ is of measure zero but can be non-empty. In \cite{Ginzburg}, the following question is raised:

\begin{Question}
  {Let $M$ be a symplectic manifold of bounded Hofer--Zehnder capacity and $H$ a smooth proper function on $M$. How large can the set $\mathcal{AP}_H$ of regular aperiodic values be?}
\end{Question}

	For a review of the known results concerning this question, see \cite{Ginzburg}. The proof of Theorem \ref{thmginzburg} is based on a plug construction.

	\subsection{Traps and Plugs}
	
By the flow-box theorem, the flow of a non-singular vector field on a $n$-dimensional manifold locally looks like the linear flow, that is: on $D^{n-1} \times [0,1]$ the flow is given by $\Psi_t: (x,s) \to (x,s+t)$, where $t\in \mathbb{R}$ and $D^{n-1}$ denotes a disk of dimension $n-1$.
	
	\begin{definition}
		A \emph{trap} is a smooth vector field on the manifold $D^{n-1} \times [0,1]$ such that
		\begin{enumerate}
			\item the flow of the vector field is given by $\frac{\partial}{\partial t}$ near the boundary of $\partial D \times [0,1]$, where $t$ is the coordinate on $[0,1]$;
			\item there are no periodic orbits contained in $D \times [0,1]$;
			\item the orbit entering at the origin of the disk $D \times \{0\} $ does not leave $D \times [0,1]$ again.
		\end{enumerate}
		If the vector field additionally satisfies \emph{entry-exit matching condition}, that is that the orbit entering at $(x,0)$ leaves at $(x,1)$ for all $x \in D \setminus \{0\}$, then the trap is called a \emph{plug}.
	\end{definition}

	As a result of the flow-box theorem, traps can be introduced to change the local dynamics of a flow of a vector field and  ``trap'' a given orbit. However, the introduction of a trap can change the global dynamical behaviour drastically. A plug additionally asks for matching condition at entry and exit in order not to change the global dynamics of the vector field. The vector field in question often satisfies some geometric properties (as for example volume-preserving, a Reeb vector field, a Hamiltonian vector field,\dots). The crux in the construction of traps and plugs is to produce a vector field satisfying the given geometric constraint. In this paper, we are going to tackle the question of existence of traps and plugs for the $b^m$-Reeb flow.
	
	Traps and plug have been successfully used to construct counter-example in existence theorem for many geometric flows. For instance, Kuperberg constructed a plug in \cite{Kuiperberg} to find a smooth non-singular vector field without periodic orbits on any closed manifold of dimension $3$. The special case of $S^3$ is known as counter-example to the Seifert conjecture.
	In the contact case, by the positive answers of Weinstein conjecture, there cannot exist plugs for the Reeb flow. Furthermore, it is a corollary of a theorem of Eliashberg and Hofer \cite{hofereliashberg} that in dimension $3$, Reeb traps do not exist. The same was conjectured in higher dimension, but Reeb traps were later proved to exist in  dimension higher than $5$, see \cite{GRZ}.

\section{On the Weinstein conjecture for $b^m$-contact manifolds}\label{sec:counterexample}
	
We will see that there are examples of compact $b^m$-contact manifolds without periodic Reeb orbits away from the critical set but that there always exists infinitely many periodic orbits on	a closed critical set in dimension $3$.

\subsection{Existence of infinitely many periodic orbits on the critical set}

As was observed in Proposition \ref{prop:ReebHamdim3}, the Reeb vector field on the critical set is a Hamiltonian vector field in the $3$-dimensional case. This is only true in dimension $3$, which comes from the fact that area forms are symplectic forms on surfaces. This will imply that on the critical set  of closed $b^m$-contact manifolds, there are infinitely many periodic Reeb orbits\footnote{The authors would like to thank Robert Cardona for pointing this out.}.

\begin{proposition}\label{prop:infinitelymanypoZ}
Let $(M,\alpha)$ be a $3$-dimensional $b^m$-contact manifold and assume the critical hypersurface $Z$ to be closed. Then there exists infinitely many periodic Reeb orbits on $Z$.
\end{proposition}

Note that the critical hypersurface $Z$ is closed if there exists a global function defining $Z$ and the ambient manifold $M$ is compact.

\begin{proof}
Let us denote the usual decomposition by $\alpha=u\frac{dz}{z}+\beta$. By Proposition \ref{prop:unonconstant3dim}, the function $u$ is non-constant on $Z$. Furthermore by Proposition \ref{prop:ReebHamdim3}, the Reeb vector field is Hamiltonian on $Z$ for the function $-u$. Let $p\in Z$ be a point such that $du_p\neq 0$. As the preimage of a closed topological set is closed and a closed set of a compact manifold is compact, the level-sets are given by circles and the Reeb vector field, contained in the level-set, is non-vanishing in view of $$\iota_{R_\alpha}(ud\beta+\beta \wedge du)=du.$$
Hence the Reeb vector field is periodic on $u^{-1}(p)$.
\end{proof}

The condition of $M$ to be closed is necessary, as it can be seen in the next example.

    \begin{example}\label{nonweinstein}
    Consider $S^2\times S^1 \setminus \{(p_N,\phi),(p_S,\phi)\}$ where $\phi$ is the angular coordinate on $S^1$ and $(h,\theta)$ are polar coordinates on $S^2$ and $p_N$ (respectively $p_S$) denotes the north pole (respectively south pole). The $b$-form  $\alpha=\frac{d\phi}{\sin \phi}+hd\theta$ is a $b$-contact form whose Reeb vector $R_\alpha=\sin \phi \frac{\partial}{\partial \phi}$ vanishes everywhere on $Z$ and does not admit any periodic Reeb orbits. However this is a non-compact example and in fact, the north and south pole cannot be added to compactify this example as this would yield a trivial Boothby--Wang fibration of $S^2$ over $S^1$.
    \end{example}

    We do not know if a similar result to the one in Proposition \ref{prop:infinitelymanypoZ} holds in higher dimension. The dimension $3$ is particular here because area forms on surfaces are symplectic.

\begin{question}
Let $(M,\alpha)$ be a closed $(2n+1)$-dimensional $b^m$-contact manifold and let $Z$ stand for the critical set. Then there exists infinitely many periodic Reeb orbits on $Z$.
\end{question}

\begin{remark}
{The induced geometry of the $b^m$-contact structure on the critical set is studied in \cite{MO} in the language of Jacobi manifolds. The critical set is foliated by even and odd-dimensional leaves, with an induced locally conformally symplectic structure in the first case (see \cite{MO} for the definition), and a contact structure in the second. In the case where $\dim M=5$ and the contact leaf is compact, it admits at least one periodic Reeb orbit in view of the positive answer to the Weinstein conjecture.}
\end{remark}

\begin{remark}\label{rmk:naivesingularWeinstein}
As in the $3$-dimensional compact case, $b^m$-contact manifolds always admit infinitely many periodic orbits on the critical set, a first \emph{naive} generalization of the Weinstein conjecture would be that there are always have periodic orbits away from the critical set. However this is not true, as we will see in the next subsection.
\end{remark}
	
\subsection{Non-existence of periodic Reeb orbits away from the critical set}

	 We will see that in the presence of singularities in the geometric structure, there are $b^m$-symplectic manifolds that do not admit any periodic orbit of the Hamiltonian vector field away from the critical set. This originates from the fact that the Weinstein conjecture stated as in Remark \ref{rmk:naivesingularWeinstein} does not hold in the singular contact set-up: there are compact $b^m$-contact manifolds $M$ with critical set $Z$ with no periodic Reeb orbits away from $Z$. In particular, we prove that taking the symplectization, there are proper Hamiltonian functions on $b^m$-symplectic manifolds having no periodic orbits for the Hamiltonian flow away from $Z\times \R$.

	\begin{claim}
		There are compact $b^m$-contact manifolds $(M,Z)$ of any dimension for all $m\in \mathbb{N}$ without periodic Reeb orbits on $M\setminus Z$.
	\end{claim}
	
	In what follows several examples where the Weinstein conjecture is not satisfied are given thus proving the claim. The first example is given by Example \ref{example:S3}.
	
	\begin{example}\label{Ex:S3nopo}
	Consider the example of the $3$-sphere in the standard $b$-symplectic Euclidean space $(\R^4,\omega)$ as in Example \ref{example:S3}. The $b$-contact form is given by $\alpha=\iota_X \omega$ where $X$ is the Liouville vector field transverse to $S^3$ given by $X=\frac{1}{2}x_1\frac{\partial}{\partial x_1}+y_1\frac{\partial}{\partial y_1}+\frac{1}{2}(x_2 \frac{\partial}{\partial x_2}+y_2\frac{\partial }{\partial y_2})$. The Reeb vector field is given by
	$$R_\alpha=2x_1^2\frac{\partial}{\partial x_1}-x_1y_1\frac{\partial}{\partial y_1}+2x_2\frac{\partial }{\partial y_2}-2y_2\frac{\partial}{\partial x_2}.$$
	On the critical set, given by $S^2$, this vector field is giving rise to rotation. Away from $Z$, the Reeb vector field does not admit any periodic orbits. Indeed, the vector field can be interpreted as two uncoupled systems in the $(x_1,y_1)$, respectively $(x_2,y_2)$-plane. The flow in the $(x_1,y_1)$-plane is clearly not periodic. 
	
	
	This example can be generalized to $b^{2k+1}$-contact forms for any $k\geq 1$ by considering $(\R^4,\omega_\text{st}= \frac{dx_1}{x_1^{2k+1}}\wedge dy_1+dx_2\wedge dy_2)$ and the Liouville vector field given by $X= \frac{1}{2}x_1^{2k+1}\frac{\partial}{\partial x_1}+y_1\frac{\partial}{\partial y_1}+\frac{1}{2}(x_2\frac{\partial}{\partial x_2}+y_2\frac{\partial}{\partial y_2})$ that is transverse to $S^3$ and hence $\alpha=\iota_X\omega$ is a $b^m$-contact form. The associated Reeb vector field does not admit any periodic orbits. The restriction on the parity comes from the fact that transversality of a similar Liouville vector field with respect to $S^3$ fails.
	\end{example}

    The next example is given by the $3$-torus, as in Example \ref{example:3torus}.
    \begin{example}\label{Ex:T3nopo}
    Consider $(\mathbb{T}^3,\sin \phi \frac{dx}{\sin x}+\cos{\phi}dy)$. The Reeb vector field is given by $R_\alpha=\sin \phi \sin{x}\frac{\partial}{\partial x}+\cos{\phi}\frac{\partial}{\partial y}$. The critical set is given by two disjoint copies of the $2$-torus $\mathbb{T}^2$ and the Reeb flow restricted to it is given by $\cos \phi \frac{\partial}{\partial y}$. As in the last example, the critical set $Z$ is given by periodic orbits (except when $\cos \phi=0$, where the Reeb vector field is singular). However, away from $Z$ there are no periodic orbits.

This example can be generalized to higher order singularities.
    \end{example}

     Armed with these examples we conclude that there are examples of $b^m$-symplectic manifolds without  periodic orbits of the Hamiltonian flow  away from the critical hypersurface.
	
\begin{claim}
    {There are $b^m$-symplectic manifolds with proper smooth Hamiltonian whose Hamiltonian flow does not have any periodic orbits away from $Z$.}
\end{claim}
	
	To see this, let $(M,\alpha)$ be a compact $b^m$-contact manifold without any periodic Reeb orbits away from the critical set and consider in its symplectization the Hamiltonian function $e^t$ (where $t$ is the coordinate in the symplectization). The Hamiltonian vector field is a reparametrization of the Reeb vector field on the level-sets and therefore provides an example of a proper smooth Hamiltonian containing no periodic orbits in the level-sets away from the critical hypersurface.

	We define the set-of aperiodic values for a $b$-symplectic manifold $(M,\omega)$ and a Hamiltonian $H\in C^\infty(M)$ as follows
$$ {^b}\mathcal{AP}_H := \{a\in \R| X_H \text{ does not admit any periodic orbits on } H^{-1}(a) \text{ away from } Z \}.$$

The corollary can be reformulated: on $b$-symplectic manifolds there is a proper smooth Hamiltonian such that  ${^b}\mathcal{AP}_H=\R$.

	 This is in stark contrast to the almost-existence theorem (Theorem \ref{almostexistence}) in symplectic geometry. To the authors knowledge, there are no known examples of Hamiltonian having no periodic orbits on all level-sets (or equivalently having $\mathcal{AP}_H=\R$).


We do not know if there are compact examples without any periodic Reeb orbits, neither away from $Z$ nor on $Z$.

    \begin{Question}
    {Are there compact $b^m$-contact manifolds  without periodic Reeb orbits neither on $Z$ nor away from $Z$?}
    \end{Question}

\begin{remark}\label{keypoint}
A $1$-dimensional closed $b$-contact manifold without any non-trivial periodic Reeb orbits is given by $(S^1,\frac{d\phi}{\sin \phi})$. In this example the singularity is transferred from the contact form to the orbit as marked points on the circle are declared as zeroes of the vector field. There are topological obstructions to generalizing this example to higher dimensions using circle actions and transferring the singularity. This follows from the remark in the Example \ref{nonweinstein}: Higher dimensional examples without any non-trivial periodic Reeb orbits cannot come from a $S^1$-action, as this would yield a trivial Boothby-Wang fibration. We still believe that enlarging the class of contact structures to $b^m$-contact structures and thereby admitting possible vanishing points of the Reeb vector field allows to construct such examples and thus a counterexample to the smooth Weinstein conjecture for $b^m$-contact structures.
\end{remark}


As mentioned before, by the positive results on the Weinstein conjecture, Reeb plugs cannot exist. However, the compact counterexamples in any dimension in the $b^m$-contact case raise the following question:
	
	\begin{Question}
	    {Are there plugs for $b^m$-Reeb flows?}
	\end{Question}

More precisely,  given a contact manifold, can the singularization  be used to change the contact structure by a $b^m$-contact structure and thereby controlling the changed Reeb dynamics to destroy a given periodic for the initial flow? A guideline example is Example \ref{Ex:S3nopo}, where the critical set $Z$ is given by a $2$-sphere and there are no periodic orbits away from $Z$. Of course, due to Proposition \ref{prop:infinitelymanypoZ}, when $Z$ is closed, there are always infinitely many periodic orbits on the critial set in dimension $3$.

In the following section, we give a first approach towards understanding the existence of plugs for the $b^m$-Reeb flow. We will see that spheres can be inserted as hypersurfaces in local Darboux charts. The dynamics on those spheres is given as in Example \ref{Ex:S3nopo} and a given orbit entering the Darboux charge is being captured as it approaches one of the fixed points of the sphere. This will yield the existence of traps for $b^m$-Reeb fields.

 \section{On the existence of traps and plugs for $b^m$-contact manifolds}\label{sec:trap}

 As mentioned before, traps for Reeb flows do not exist in dimension $3$, see \cite{hofereliashberg}, but do exist in higher dimensions \cite{GRZ}. We prove that in the $b$-category, there is no restriction on the dimension.

	 \begin{theorem}
	 There exist $b^m$-contact traps in any dimension.
	 \end{theorem}
	
	 \begin{proof}
	 We only consider the $3$-dimensional case, the higher dimensional being similar. Consider a Darboux ball and denote the standard contact form by $\alpha_{st}$. For convenience, we work in polar coordinates, in which $\alpha_{st}$ writes down as $\alpha_{st}=dz+r^2d\theta$. The Reeb vector field is given by $R_\alpha=\frac{\partial}{\partial z}$.
	
	 We introduce a convex hypersurface and use the existence result, Theorem \ref{thm:existenceb2k}. The vector field $X=2z\frac{\partial}{\partial z}+r\frac{\partial}{\partial r}$ is a contact vector field as $\mathcal{L}_X\alpha=2\alpha$ and is transverse to the $2$-sphere $S^2$. Hence $S^2$ is a convex hypersurface and can be realized as critical set of a $b^{2k}$-contact manifold. More precisely, introducing in a neighbourhood around $S^2$ the coordinate $t$ such that $X=\frac{\partial}{\partial t}$, the contact form writes in the Giroux decomposition as follows:
	 $$\alpha=g(udt+\beta)$$
	 where $u\in C^\infty(S^2)$, $\beta\in \Omega^1(S^2)$ and $g$ is a smooth function. Note that $u$ and $\beta$ are independent of the $t$-coordinate, whereas $g$ is not. The $b^{2k}$-form is given by
	 $$\alpha_\epsilon=g(udf_\epsilon+\beta).$$
	 In order to compute the Reeb dynamics, let us explicitely compute the functions $u,g$ and the $1$-form $\beta$.
	 We introduce the following change of variables:
	 $$\begin{cases}
	 \frac{\partial}{\partial t}= 2z\frac{\partial}{\partial z}+r\frac{\partial}{\partial r} \\
	 \frac{\partial}{\partial \xi}=r^2 \frac{\partial}{\partial z}.
	 \end{cases}$$
	 The dual basis is given by
	$$\begin{cases}
	dt=\frac{1}{r}dr \\
	d\xi=\frac{1}{r^2}dz-\frac{2z}{r^3}dr.
    \end{cases}$$
    Hence $r=e^t$ and $z=e^{2t}\xi$.
    Under this change of variable, the Giroux decomposition of the standard contact form is given by
    $$\alpha_{st}=e^{2t}(2\xi dt+d\xi+d\theta).$$
    The $b^{2k}$-contact form is given by
     $$\alpha_\epsilon=e^{2t}(2\xi df_\epsilon+d\xi+d\theta)$$ and a direct computation yields that the Reeb vector field associated to $\alpha_\epsilon$ is given by
     $$R_{\alpha_\epsilon}=\frac{f_\epsilon'-1}{f_\epsilon'}e^{-2t}\frac{\partial}{\partial \theta}+\frac{1}{f_\epsilon'}e^{-2t}\frac{\partial}{\partial \xi}.$$
	 Close to the singularity, $f_\epsilon'=\frac{1}{t^2}$ so that on the critical set $S^2=\{t=0\}$, the Reeb dynamics is given by $R=\frac{\partial}{\partial \theta}$. Furthermore, it follows from the formulas that the orbit entering the Darboux ball at $\theta=0$ limits to the fixed point on $S^2$ and hence is trapped. See Figure \ref{fig:trap}, where the dynamics around $Z=S^2$ and the trapped orbit is depicted.

\begin{center}

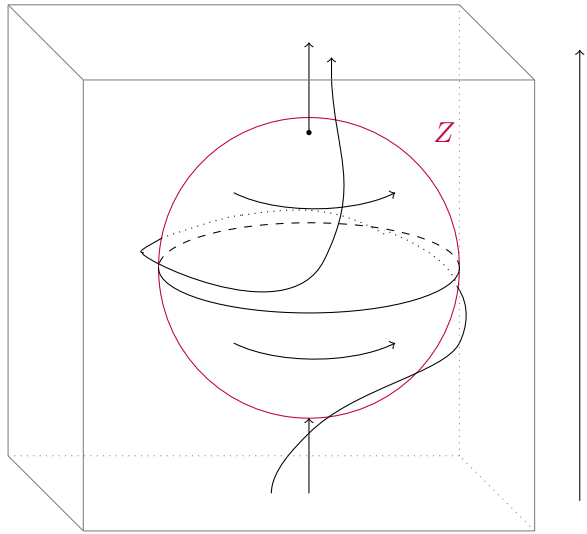
\begin{figure}[hbt!]

\begin{tikzpicture}


  \draw[purple] (0,0) circle (2cm);

  \draw[purple] (1.8,1.8) node {$Z$};


  \draw (-2,0) arc (180:360:2 and 0.6);

  \draw[dashed] (2,0) arc (0:180:2 and 0.6);


  \draw [<-] (0,-2) -- (0,-3);

  \draw [->] (0,1.8) -- (0,3);

  \fill[fill=black] (0,1.8) circle (1pt);

  \draw [->] (-1,-1) arc (220:320:1.4 and 0.6);

  \draw [->] (-1,1) arc (220:320:1.4 and 0.6);


  \draw (-0.5,-3) ..controls +(0,0.3) and +(-0.2,-0.2).. (0,-2.2);

  \draw (0,-2.2) ..controls +(0.6,0.6) and +(-0.2,-0.4).. (2,-1);

  \draw (2,-1) ..controls +(0.1,0.2) and +(0.18,-0.25).. (1.97,-0.25);

  \draw [dotted](1.97,-0.25) ..controls +(-0.08,0.4) and +(0.15,-0.05).. (1,0.5);


  \draw [dotted] (1,0.45) ..controls +(-0.7,0.4) and +(0.7,0.1).. (-0.9,0.7);

   \draw[dotted] (-0.9,0.7) ..controls +(-0.2,0) and +(0,0.0005).. (-1.95,0.4);

	\draw (-1.95,0.4) ..controls +(-0.02,-0.02) and +(-0.15,0).. (-2.2,0.2);



  \draw (-2.2,0.2) ..controls +(-0.2,0) and +(-0.5,-1).. (0.2,0.1);

   \draw  (0.2,0.1) ..controls +(0.5,1) and +(0.05,-1).. (0.3,2.5);

   \draw[->] (0.3,2.5) -- (0.3,2.8);















\begin{scope}[shift={(0,-0.5)}]

	\draw[gray] (-3,-3) -- (-3,3);

	\draw[gray] (-3,-3) -- (3,-3);

	\draw[gray] (3,-3) -- (3,3);

	\draw[gray] (-3,3) -- (3,3);

	\draw[gray] (-3,-3) -- (-4,-2);

	\draw[gray,dotted] (3,-3) -- (2,-2);

	\draw[gray] (-3,3) -- (-4,4);

	\draw[gray] (3,3) -- (2,4);

	\draw[gray] (-4,-2) -- (-4,4);

	\draw[gray] (-4,4) -- (2,4);

	\draw[gray,dotted] (2,4) -- (2,-2);

	\draw[gray,dotted] (2,-2) -- (-4,-2);

	\draw[->] (3.6,-2.6) -- (3.6,3.4);

\end{scope}

\end{tikzpicture}
\caption{The $b^{2k}$-contact trap. The critical set $Z$ is given by $S^2$. Outside of the $\epsilon$-neighbourhood the dynamics is unchanged. One Reeb orbit is trapped. One orbit is depicted that does not satisfy the entry-exit condition.}
\label{fig:trap}
\end{figure}
\end{center}

	 A similar proof holds for the case of $b^{2k+1}$-contact structures, where there are two disjoint copies of $S^2$ being identified as critical set to bypass the orientation issues.
		 \end{proof}

	 Plugs for the Reeb flow cannot exist by the positive answer to Weinstein conjecture. In the case of $b^m$-contact manifolds, we exhibited examples in the last section of compact $b$-contact manifolds without any periodic Reeb orbits away from $Z$. Although we were only able to construct a trap for the Reeb flow on $b^m$-contact manifolds, we strongly believe plugs exist in the singular context. In the last proof, due to the term in $\frac{\partial}{\partial \theta}$, the entry-exit condition is not satisfied, meaning that this construction only yields a trap and not a plug.
	
	 Note that this construction is narrowly related to Example \ref{Ex:S3nopo}. The dynamics on $Z$ of this example and of the construction in the last proof are the same. In contrast to the trap construction, there are no periodic orbits for the example on $S^3$. This corroborates our view on the existence of plugs for  $b^m$-Reeb flows.

\section{Overtwisted singular contact manifolds and non-compact contact manifolds}\label{sec:bmhofer}

While earlier sections dealt with the non-existence of periodic Reeb orbits away from the critical set, we will now see a sufficient condition to guarantee the existence of periodic orbits away from the critical set. The techniques are based on Hofer's proof of the Weinstein conjecture for overtwisted contact manifolds, recalled in Subsection \ref{subsec:reviewweinstein}.

In this section, we consider $b^m$-contact manifolds that have an overtwisted disk away from the critical set. In particular, we consider the manifold in this section to be of dimension $3$. For the definition of higher dimensional overtwisted contact manifolds see \cite{niederkruger,bem}.

\begin{definition}\label{def:bovertwisted}
We say that a $b^m$-contact manifold is overtwisted if there exists an overtwisted disk away from the critical hypersurface $Z$.
\end{definition}

We will prove that in the case of overtwisted $b^m$-contact manifolds Weinstein conjecture holds, under the supplementary condition that $\alpha$ is $\R^+$-invariant around the critical hypersurface. Before rigorously definining this $\R^+$-invariance, let us first state the main theorem of this section. We will prove the following:

	\begin{theorem}\label{thm:bot}
	  Let $(M,\alpha)$ be a closed $b^m$-contact manifold with critical set $Z$. Assume there exists an overtwisted disk in $M \setminus Z$ and assume that $\alpha$ is $\R^+$-invariant in a tubular neighbourhood around $Z$. Then there exists
	  \begin{enumerate}
	      \item   a periodic Reeb orbit in $M \setminus Z$ or
	      \item a family of periodic Reeb orbits approaching the critical set $Z$.
	  \end{enumerate}
	  Furthermore, the periodic orbits are contractible loops in the symplectization.
	\end{theorem}
	
	\begin{remark}
	The condition of $\alpha\in {^{b^m}}\Omega^1(M)$ to be $\R^+$-invariant is a non-trivial condition, as is pointed out in Remark \ref{rmk:Rinvariance}.
	\end{remark}
	
	The proof is based on Hofer's original arguments. The novelty here is that we work in a non-compact set-up, namely on the open manifold $M\setminus Z$. On open manifolds, Hofer's method generally does not apply. However the openness is \emph{gentle} due to the $\R^+$-action. We will see that this theorem is a corollary of a more general statement: in fact, we do not need the geometric structure to be $b^m$-contact, we only need a contact form on an open manifold that is $\R^+$-invariant in the open ends of the manifold and overtwisted away from the $\R^+$-invariant part, see Theorem \ref{thm:R^+invariantHofer}.
	
Under those assumption, in the $\R^+$-invariant part of the manifold, the decomposes as the union of products of the connected components of the compact boundary with $\R^+$. We will see that in this decomposition pseudoholomorphic curves can be translated in the $\R^+$-direction and the compactness of the boundary of compact set guarantees convergence. With this in mind, we define the following:

\begin{definition}\label{def:R^+invariantcontact}
Let $\alpha\in \Omega^1(M)$ be a contact form on an open manifold $M$. We say that $\alpha$ is \emph{$\R^+$-invariant} in the open ends of $M$ if there exists a compact set $K\subset M$ and a vector field $X$ defined on $M \setminus K$ that satisfies $\mathcal{L}_X \alpha=0$, meaning that $X$ is a strict contact vector field and such that $X$ is transverse to $\partial K$. We say that an $\R^+$-invariant contact form $\alpha$ is overtwisted if there is an overtwisted disk  contained in $K$.
\end{definition}

\begin{remark}\label{rmk:Rinvariance}
In the light of Definition \ref{def:R^+invariantcontact}, we will view $b^m$-contact manifolds as open contact manifolds by considering the manifold without the critical set. The $\R^+$-invariance then translates into the fact that the $b^m$-contact form admits a decomposition given by
$\alpha=u\frac{dz}{z^m}+\beta$, where $u\in C^\infty(Z)$ and $\beta\in \Omega^1(Z)$. This decomposition was already studied thoroughly in \cite{MO} in the problem of existence of $b^m$-contact structures under the name of \emph{convex} $b^m$-contact forms (see Definition \ref{def:bcontactconvex}). Not every $b^m$-contact form is of course $\R^+$-invariant: for example the kernel of $g\alpha$, where $g\in C^\infty(M)$ is positive, defines the same contact structure as $\ker \alpha$, but a priori, there is no reason for the function $g$ to be $\R^+$-invariant.
\end{remark}

\begin{example}
An example of an $\R^+$-invariant $b^m$-contact form is given by Example \ref{example:3torus}. Consider $(\mathbb{T}^3,\alpha= \sin \phi \frac{dx}{\sin x}+\cos \phi dy)$. Indeed, the vector field given by $\sin x\frac{\partial}{\partial x}$ is a strict contact vector field (that is it satisfies $\mathcal{L}_X\alpha=0$) and it is transverse to the critical set. However, this example is not overtwisted: there are no periodic Reeb orbits away from the critical set as we saw in Example \ref{Ex:T3nopo}.
\end{example}

\begin{example}
The $b$-contact form on $S^3$ given by $\alpha=\iota_X \omega$ exhibited in Example \ref{example:S3} is not $\R^+$-invariant around the critical set. The Reeb vector field is not transverse to the critical set and also, in the decomposition of the $b$-contact form, it clearly does not decompose as $\alpha=u\frac{dz}{z}+\beta$ where $u\in C^\infty(Z)$ and $\beta\in \Omega^1(Z)$.
\end{example}

The flow of the vector field $X$ generates a $\R^+$-action. We often refer to $M \setminus K$ as the $\R^+$-invariant part of the contact manifold $(M,\alpha)$ and will denote it by $M_{inv}$.

In the $\R^+$-invariant part of the contact manifold $M_{inv}$, we have coordinates adapted to the action. Indeed, by following the flow of $X$, the $\R^+$-invariant part is diffeomorphic to $\partial K \times \R^+$, where $K$ is the compact set with boundary as in Definition \ref{def:R^+invariantcontact}.

For our purposes, we fix notation for those coordinates in the case of maps from the disk $D$ to $M$. Consider $u: D\to M$ and assume that for $z\in D$, $u(z)\in M_{inv}$. We will write
\begin{equation}\label{eq:R^+invariantcoordinates}
u(z)=(d(z),w(z))
\end{equation}
where $d(z)\in \R^+$ and $w(z)\in \partial K$. This notation will appear in the proof of the main theorem of this section, which is given by the following.

\begin{theorem}\label{thm:R^+invariantHofer}
Let $(M^3,\alpha)$ be an overtwisted $\R^+$-invariant contact manifold. Then there exists a $1$-parametric family of periodic Reeb orbits in the $\R^+$-invariant part of $M$ or a periodic Reeb orbit away from the $\R^+$-invariant part.
\end{theorem}

It is clear that Theorem \ref{thm:bot} is a straightforward corollary of this theorem. Furthermore, Theorem \ref{thm:R^+invariantHofer} applies to more general open contact manifolds as is shown in the next example.

\begin{example}\label{ex:ccontact}
Consider the $3$-torus $\mathbb{T}^3$ and consider the non-smooth $1$-form given by $\alpha=\sin \phi \frac{dx}{\sin x}+\cos \phi \frac{dy}{\sin y}$. The vector field $\sin x \frac{\partial}{\partial x}+\cos y \frac{\partial}{\partial y}$ is a strict contact vector field and is transverse to the boundary of a tubular neighbourhood of $\mathbb{T}^3 \cap \{x=0\} \cap \{y=0\}$ and therefore $\alpha$ is a $\R^+$-invariant contact form, but is not a $b$-contact form.
\end{example}

We now sketch the proof of Theorem \ref{thm:R^+invariantHofer}. As in the standard setting, we study the Bishop family emanating from the overtwisted disk as in Theorem \ref{thm:bishop} and we aim to prove the existence of a finite energy plane and conclude by Theorem \ref{thm:finiteenergy}. As in the standard case, the gradient blows up in the interior of the disk. However, in contrast to the standard proof, we distinguish two different cases.

In the first case, we assume that the sequence where the gradient blows up is contained in a bounded subset of $M$. In this case, the standard arguments apply, and there a reparametrization of the bubble yields a finite energy plane contained in the symplectization of the bounded subset of $M$. This yields the existence of a periodic Reeb orbit away from the $\R^+$-invariant part.

In the opposite case, the sequence of points where the gradient blows up is not bounded in $M$. Loosely speaking this means that the point of blow-up diverges in the non-compact $\R^+$-invariant part. This non-compactness behaviour is settled by translating the $J$-holomorphic curves in the direction of the $\R^+$-action and therefore, in the decomposition as in Equation \ref{eq:R^+invariantcoordinates}, the first term is being kept constant (so it trivially converges) and the second term is contained in the compact set $\partial K$, so Arzela--Ascoli theorem applies to this term. We thereby obtain a sequence of $J$-holomorphic disks, contained in the symplectization of the $\R^+$-invariant part of $M$, that converge to a non-trivial finite energy plane.
This yields a periodic orbit in the $\R^+$-invariant part and by $\R^+$-invariance therefore also a $1$-parametric family of periodic Reeb orbits.

\begin{remark}
In \cite{pp}, the authors prove a foliated version of the Weinstein conjecture. Similar to the main theorem of this section (Theorem \ref{thm:R^+invariantHofer}), the {leaves} of the foliated space are not assumed to be compact, however compactness of the ambient space ensures that Arzela--Ascoli theorem can be applied.
\end{remark}
	
We begin by collecting the necessary lemmas to prove the main theorem. The proof of Theorem \ref{thm:R^+invariantHofer} is then done in the subsequent subsection, Subsection \ref{subsec:mainthm}.

\subsection{Necessary lemmas}

We denote the overtwisted disk, that exists away from the $\R^+$-invariant part, by $D_{OT}$ and denote its elliptic singularity by $e$. As usual $D_{OT}^*$ denotes $D_{OT}\setminus \{e\}$.

Compactness or non-compactness of the Bishop family $\{\u_t\}$ depends on the whether or not the gradient $\nabla \u_t$ is uniformly bounded.

\begin{lemma}[Uniform gradient bound implies compactness of family]\label{lem:uniformboundimpliescompact}
 Let $(M,\alpha)$ be a $\R^+$-invariant contact manifold. Assume that
 $$\widetilde{u}_t=(a_t,u_t): D \to \R\times M, t \in [0,\epsilon)$$
 is a Bishop family as in Theorem \ref{thm:bishop} satisfying the boundary conditions as in Equation \ref{eq:boundaryconditions}. Assume furthermore that
 \begin{equation}\label{def:gradientblowup}
 \sup_{0\leq t<\epsilon} \norm{\nabla \widetilde{u}_t}_{C^0(D)} < \infty.
 \end{equation}
 Then $ \widetilde{u}_t \to  \widetilde{u}_\epsilon$ in $C^\infty(D)$ as $t\to \epsilon$ where $ \widetilde{u}_\epsilon$ is an embedded pseudoholomorphic disk satisfying $ \widetilde{u}_\epsilon(\partial D)\subset \{0\}\times D_{OT}^*$.
 \end{lemma}

 \begin{proof}
     The proof can be found in Proposition 8.1.2 in \cite{abbashofer}.
 \end{proof}

Let us first do a remark concerning gradient blow-ups.

\begin{remark}
As stated in the last lemma, compactness or non-compactness of the Bishop family $\{\u_t\}$ depends on the whether or not the gradient $\nabla \u_t$ is uniformly bounded. Let us remark here that by blow-up of the gradient, we mean that the gradient of $J$-holomorphic families of disk blows up for a certain parametrization of the disk.

More precisely, it is clear that if $\u_t:D\to \R\times M$ is a family of $J$-holomorphic disks satisfying the usual boundary conditions
\begin{equation}\label{eq:boundaryconditions}
\u_t(\partial D) \subset \{0\} \times D_{OT}^* \text{ and } \inf_{0\leq t<\epsilon} \text{dist}(u_t(\partial D), e)>0
\end{equation}
and $\phi:D\to D$ is a conformal automorphism of the unit disk, then $\u_t\circ \phi$ is also a $J$-holomorphic curve satisfying the same boundary conditions. It is well-known that the conformal automorphism group is of the disk is non-compact and generated by
$$\phi(z)=e^{i\alpha} \frac{a-z}{1-\overline{a}z}, $$
where $\alpha \in [0,2\pi)$ and $a\in \text{int}(D)$. By choosing $a$ close to $\partial D$, $\norm{\nabla \phi}_{C^0(D)}$ becomes arbitrarily large. Therefore, when we say that the gradient blows up, we mean that it blows up for the infimum of all possible conformal reparametrization of the disk. More explicitly, we mean that the quantity
$$e(\u_t):=\inf_{\phi \in \text{Aut}(D)} \norm{\nabla (\u_t\circ \phi)}_{C^0(D)}$$
goes to infinity when $t\to \epsilon$.
\end{remark}

	As in \cite{abbashofer}, the energy of the family $ \widetilde{u}_t$ is bounded above by the $d\alpha$-area of the overtwisted disk.

\begin{lemma}[Universal upper bound on the energy]\label{lem:upperboundenergy}
 Let $(M,\alpha)$ be a $\R^+$invariant contact manifold. Assume that
 $$ \widetilde{u}_t=(a_t,u_t): D \to \R\times M, t \in [0,\epsilon)$$
 is a Bishop family as in Theorem \ref{thm:bishop} satisfying $$\inf_{0\leq t <\epsilon} \text{dist}(u_t(\partial D),e)>0.$$
 Then there exists a constant $C=C(\alpha,D_{OT})>0$ such that $E( \widetilde{u}_t)<C$.
\end{lemma}

 \begin{proof}
     See Lemma 8.1.3 in \cite{abbashofer}.
 \end{proof}
	
Similar to the arguments in \cite{abbashofer}, the horizontal energy of the family $\{\u_t\}_t$ is bounded below independently of $t$. Note that a priori, compactness is needed in order to apply Arzela--Ascoli theorem. However, the limit of the  Bishop family in the non-compact manifold $M$ is taken care of by the $\R^+$-invariance.

\begin{proposition}[Universal lower bound on the horizontal energy]\label{lem:lowerboundenergy}
Let $(M,\alpha)$ be a $\R^+$-invariant contact manifold. Assume that
 $$ \widetilde{u}_t=(a_t,u_t): D \to \R\times M, t \in [0,\epsilon)$$
 is a Bishop family as in Theorem \ref{thm:bishop} satisfying  $$\inf_{0\leq t <\epsilon} \text{dist}(u_t(\partial D),e)>0.$$
 Then there exists a constant $c>0$ independent of $t$ so that
$$E^h(u_t):=\int_D u_t^*d\alpha \geq c.$$
\end{proposition}


\begin{proof}
The proof follows the strategy of the proof of Proposition 8.1.4 in \cite{abbashofer}.
Let us argue by contradiction. We pick a sequence $\{\u_{t_k}\}_k$ such that $t_k\to \epsilon$ as $k\to \infty$. For convenience, we write $\u_k=\u_{t_k}$.
We assume by contradiction that $\int_D u_k^*d\alpha\to 0$ as $k\to \infty$. There are two possibilities: either the gradient is uniformly bounded or it is not.

\underline{Case I}: the gradient is uniformly bounded.

The case where the gradient is uniformly bounded is as in proof of Proposition 8.1.4: $\u_k $ converges in $C^\infty$ to some $\widetilde{v}$ by Lemma \ref{lem:uniformboundimpliescompact} which satisfies that $E^h(v)=0$. Following the original proof, this implies that $\widetilde{v}$ is constant, which is a contradiction with $\widetilde{v}$ having non-zero winding number.

\underline{Case II}: the gradient blows up.

Let us take a sequence $z_k\in D$ such that
$$R_k:= |\nabla \u_{k}(z_k)| \to \infty$$
as $k\to \infty$ and let us assume that $z_k\to z_0\in D$ (after passing maybe to a subsequence). We will do the bubbling analysis à la Sacks--Uhlenbeck around the point $z_0$ to derive a contradiction with the assumption that $\int_D u_k^* d\alpha\to 0$. Due to the non-compactness of $M$, care needs to be taken in the bubbling analysis. Due to the $\R^+$-invariance, the non-compactness is \emph{mild} and we will see that Arzela--Ascoli theorem can still be applied to show convergence to some $J$-holomorphic plane (in Subcase I) respectively to some $J$-holomorphic half plane (in Subcase II).

Take a sequence of $\epsilon_k>0$ such that $\epsilon_k \to 0$ and $R_k\epsilon_k \to \infty$. By Hofer's lemma (Lemma 4.4.4 in \cite{abbashofer}), we can additionally assume that
\begin{equation}
    |\nabla \u_{k}(z)| \leq 2R_k \quad \text{if } |z-z_k|\leq \epsilon_k.
\end{equation}
We distinguish two subcases.

\underline{Subcase I}: the gradient blows up in the interior.

More precisely, by this we mean that $R_k\text{dist}(z_k,\partial D) \to \infty$. We will show that in this case a non-trivial finite energy plane bubbles off and has zero horizontal energy, which leads to a contradiction.

We distinguish two further subcases, depending whether or not the gradient blows up in the compact subset $K \subset M$ or the $\R^+$-invariant part $M\setminus K$.

First, let us assume that $u_k(z_k)$ remains in the compact subset $K$. Then the standard arguments as in Proposition 8.1.4 in \cite{abbashofer} apply and lead to a non-trivial finite energy plane having zero horizontal energy, which is a contradiction.

Hence, we assume that $ u_k(z_k)$ tends to the $\R^+$-invariant part. We use the coordinates introduced in Equation \ref{eq:R^+invariantcoordinates},
$$\u_k=(a_k,d_k,w_k)$$
and we assume that $d_k(z_k) \to \infty$. If this was not that case, we would be in the case where $u_k(z_k)$ remains in the compact subset $K$.

We now define the following pseudoholomorphic curves, which are a translation in the invariant direction of a reparametrization of $\u_k$. We define $v_k(z)=u_k(z_k+\frac{z}{R_k})$ so that $v_k(0)=u_k(z_k)$ is the point where the gradient blows up. For $z\in B_{R_k}(-R_kz_k) \cap v_k^{-1}(M_{inv})$, we define
$$\widetilde{v}_k(z)=\big(a_{k}\big(\frac{z}{R_k}+z_k \big)-a_{k}(z_k),d_{k}\big(z_k+\frac{z}{R_k}\big) -d_{k}(z_k)+N,w_k\big(z_k+\frac{z}{R_k}\big)\big)$$
and denote the components by $\widetilde{v}_k:=(e_k,f_k,q_k)$ where $(f_k,q_k)\in \partial K \times \R^+$ are the $\R^+$-invariant coordinates. The map $(f_k,q_k)$ is a translation in the $\R^+$-invariant direction of $v_k$ that we do in order to apply Arzela--Ascoli to the function $q_k$ that is contained in the compact space $\partial K$.

The family $J$-holomorphic curves $\widetilde{v}_k$ satisfies
\begin{enumerate}
    \item $|\nabla \widetilde{v}_k(0)|=1,$
    \item $|\nabla \widetilde{v}_k(z)| \leq 2$ if $z\in \Omega_k:=B_{\epsilon_kR_k(0)}(0)\cap B_{R_k}(-R_kz_k)\cap v^{-1}_k(M_{inv}),$
    \item $e_k(0)=f_k(0)=0$.
\end{enumerate}
Furthermore
    $$\int_{B_{R_k}(-R_kz_k)\cap v^{-1}(M_{inv})} (f_k,q_k)^*d\alpha \leq \int_D u_{k}^*d\alpha \to 0$$
when $k\to \infty$.

We claim that $\cup_{k} \Omega_k=\mathbb{C}$. As $z\in B_{\epsilon_k R_k}$, $|z_k-\frac{z}{R_k}-z_k|<\epsilon_k$. As $v_k(z_k)\in M_{inv}$, we thus obtain that for $k$ large, $v_k(\Omega_k)\subset M_{inv}$. Hence for $k>k_0$, where $k_0$ is large, $\Omega_k= B_{\epsilon_kR_k(0)}(0)\cap B_{R_k}(-R_kz_k):=\widetilde{\Omega}_k$. As $R_k \text{dist}(z_k, \partial D)\to \infty$, we have that $\cup_{k\geq k_0}\widetilde{\Omega}_k=\mathbb{C}$.

  By the $C^\infty_\text{loc}$-bounds (Theorem 4.3.4 in \cite{abbashofer}), we conclude that (up to choosing a subsequence) $\widetilde{v}_k$ converges in $C^\infty_{\text{loc}}$ to a $J$-holomorphic plane
    $$\widetilde{v}=(b,v): \mathbb{C} \to \R \times M_{inv}$$
    satisfying $|\nabla \widetilde{v}(0)|=1$, $\int_{\mathbb{C}}v^* d\alpha=0$. Furthermore $E(\widetilde{v})\leq C$.
    Indeed, let $Q\subset \mathbb{C}$ be a compact subset and assume $k>k_0$ so that $v_k(\Omega_k)\subset M_{inv}$. We compute
    \begin{align}
        \sup_{\phi \in \mathcal{C}} \int_Q \widetilde{v}^*_kd(\phi\alpha) \leq \sup_{\phi \in \mathcal{C}} \int_{B_{R_k}(-R_kz_k)}M \widetilde{v}^*_kd(\phi \alpha) \\
        =\sup_{\phi \in \mathcal{C}}\int_D \widetilde{u}^*_k d(\phi \alpha)=E(\widetilde{u}_k)\leq C,
    \end{align}
where $C>0$ is such as in Lemma \ref{lem:upperboundenergy}. We now take the limit $k\to \infty$ and then the supremum over all compact sets $Q\subset \mathbb{C}$ to obtain $E(\widetilde{v})<\infty$.

    By Proposition 4.4.2 in \cite{abbashofer}, $\widetilde{v}$ is constant, which is in contradiction with $|\nabla \widetilde{v}(0)|=1$. We conclude that Subcase I cannot happen.

\underline{Subcase II}: the gradient blows up on the boundary.

More precisely, by this we mean that $R_k\text{dist}(z_k,\partial D) \to l\in [0,\infty)$. We furthermore assume that the gradient only blows up at the boundary, that is that there does not exist another subsequence such that $R_k\text{dist}(z_k,\partial D) \to \infty$. If this was the case, we would be in Subcase I, and therefore obtain a contradiction. Due to the boundary condition $u_t(\partial D)\subset D_{OT}$, we can assume that the subsequence $u_t$ is therefore contained in the compact subspace $K\subset M$ and therefore the standard arguments of the proof of Proposition 8.1.4 in \cite{abbashofer} apply and shows bubbling of a non-trivial finite energy half-plane. This is in contradiction with the horizontal energy to be zero.

This finishes the proof of Proposition \ref{lem:lowerboundenergy}.
\end{proof}


Finally, the next lemma assures that bubbling does not happen at the boundary of the holomorphic disks.

\begin{proposition}[Forbidding bubbling at the boundary]\label{prop:noboundarybubbling}
Let $ \{\widetilde{u}_t\}_t$ a Bishop family as before. Assume that the gradient blows up, that is
$$\sup_{0\leq t \leq \epsilon} e( \widetilde{u}_t)=\infty$$
where $e(\u_t)$ is given as in Equation \ref{eq:boundedgradient}. If $t\to \epsilon$ and $(z_k)_{k\in \mathbb{N}} \subset D$ are sequences so that $R_k:= |\nabla \widetilde{u}_{t_k}(z_k)| \to \infty$. Then the sequence $R_k\text{dist}(z_k,\partial D)$ is unbounded.
\end{proposition}
\begin{proof}
    See Proposition 8.2.1 in \cite{abbashofer}.
\end{proof}
	
\subsection{Proof of Theorem \ref{thm:R^+invariantHofer}}\label{subsec:mainthm}

We continue by proving the main theorem, that is Theorem \ref{thm:R^+invariantHofer} using the collection of lemmas and propositions of the last subsection.

\begin{proof}[Proof of Theorem \ref{thm:R^+invariantHofer}]
    Consider the Bishop family $\{\widetilde{u}_t\}_{t\in[0,\epsilon)}$ emanating from the overtwisted disk as in Theorem \ref{thm:bishop}. Note that the loops $ \widetilde{u}_t(\partial D)$ never intersect $\partial D_{OT}$ in view of Lemma \ref{lem:transverse}. If $e( \widetilde{u}_t)$ was bounded, then by Lemma \ref{lem:uniformboundimpliescompact} we could continue the family $\{ \widetilde{u}_t\}_t$ beyond $\epsilon$ which contradicts maximality of the family. Hence $e( \widetilde{u}_t)$ is unbounded.
    Pick a sequence $ \{ \widetilde{u}_k\}_k$ such that $\norm{\nabla  \widetilde{u}_k}_{C^0(D)} \to \infty$ as $k \to \infty$. We know that there are lower and upper bounds for the energy by Lemma \ref{lem:upperboundenergy} and Proposition \ref{lem:lowerboundenergy}:
    $$ c\leq E( \widetilde{u}_k) \leq C.$$
    Pick a sequence of points $(z_k)_{k\in \mathbb{N}} \subset D$ so that $R_k := | \nabla  \widetilde{u}_k(z_k)| \to \infty$. By Proposition \ref{prop:noboundarybubbling}, bubbling on the boundary cannot happen and we therefore can assume (after passing to a subsequence) that $z_k \to z_0$ and $R_k \text{dist}(z_k,\partial D)\to \infty$ when $k\to \infty$.

    Here is where the main difference with Hofer's standard proof occurs: There are two possibilities: either the gradient blows up away from the $\R^+$-invariant part or it blows up inside the $\R^+$-invariant part. This was observed already in Proposition \ref{lem:lowerboundenergy} and we repeat similar arguments here.

    For the map ${u}_k(z_k)$, this is saying that either for $k$ large (up to a subsequence to avoid mixed behaviour), $u_k(z_k)$ can satisfy one of the two cases:

    \begin{enumerate}
        \item $u_k(z_k)$ remains away from the $\R^+$-invariant part. In this case, the standard arguments apply as we can assume that $u_k(D) \subset K$, where $K$ is compact (the set $K$ is as in Definition \ref{def:R^+invariantcontact}). Therefore Arzela--Ascoli theorem applies, as well as the rest of Hofer's arguments. This proves that there exists a periodic orbit away from the $\R^+$-invariant part. This proves the first part of the theorem.
        \item $u_k(z_k)$ tends to the $\R^+$-invariant part. In this case, as mentioned before, we run into compactness issues and therefore the standard arguments do not apply directly.
    \end{enumerate}

    In what follows, we hence assume that for $k\to \infty$, $u_k(z_k)$ tends to the $\R^+$-invariant part and we will prove that this implies that there is a $1$-parametric family of periodic orbits in the $\R^+$-invariant neighbourhood.

    More precisely, we assume that for all $k>k_0$, $u_k(z_k)=(d_k(z_k), w_k(z_k))$ as in Equation \ref{eq:R^+invariantcoordinates}. Furthermore, we assume without loss of generality that $d_k(z_k)\to \infty$. Indeed, if this was false, we could just enlarge the compact set $K$ and we would be in the first case.

    We will do now the bubbling analysis around the point $z_k$ as in Hofer and in order to apply Arzela--Ascoli theorem, we do a translation in the $\R^+$-invariant direction.

    Take a sequence $\epsilon_k\to 0$ so that $R_k\epsilon_k \to \infty$. We now use the so called Hofer's lemma (Lemma 4.4.4 in \cite{abbashofer}) to additionally assume that
    \begin{equation}\label{eq:boundedgradient}
    | \nabla  \widetilde{u}_k(z)| \leq 2 R_k
    \end{equation}
    for all $z \in D$ with $|z-z_k|\leq \epsilon_k$. We define for $z \in B_{R_k}(-R_kz_k)$ the pseudoholomorphic maps $ \overline{v}_k=(b_k,v_k)$ given by
    $$\overline{v}_k(z):= \big( a_k(z_k+\frac{z}{R_k})-a_k(z_k),u_k(z_k+\frac{z}{R_k})\big).$$

    In the standard case (so also in the first case higher up), these maps are shown to converge to a non-constant finite energy plane using Arzela--Ascoli theorem. However, in this case, Arzela--Ascoli theorem does not apply because $u_k$ is not contained in a compact space. More precisely, $v_k(0)=u_k(z_k)$ and therefore $v_k$ is contained in a $M_{inv}$ around the origin. To overcome this, we define the following pseudoholomorphic maps, which is just a translation of the previous one in the $\R^+$-invariant direction. For $z\in B_{R_k}(-R_kz_k) \cap v_k^{-1}(M_{inv})$, we define:
    $$ \widetilde{v}_k(z):= \big(a_k(z_k+\frac{z}{R_k})-a_k(z_k),d_k(z_k+\frac{z}{R_k})-d_k(z_k), w_k(z_k+\frac{z}{R_k}) \big).$$

    We do this translation in the $\R^+$-invariant direction because $d_k(z_k)\to \infty$. We will now prove that $\widetilde{v}_k$ converges in $C^\infty_{loc}$ to a non-trivial finite energy plane.

      Let us denote the components of the map $ \widetilde{v}_k=(e_k,v_k)=(e_k,f_k,q_k)$

    It is clear from the reparametrization that
    $$ e_k(0)=0, f_k(0)=0, \text{ and } |\nabla  \widetilde{v}_k(0)|=1. $$
    The advantage of the reparametrization $\widetilde{v}_k$ with respect to the one given by $\overline{v}_k$ is that the convergence of $f_k$ is being taken care of as it is fixed at the origin and Arzela--Ascoli theorem applies to $q_k(z)$ as it belongs to the compact set $\partial K$. This was not the case for $\overline{v}_k$.

    Consider the domains $\Omega_k:= B_{R_k}(-R_kz_k) \cap B_{\epsilon_kR_k}(0)\cap v_k^{-1}(M_{inv})$.

    For $k$ sufficiently large, we claim that $v_k(\Omega_k) \subset M_{inv}$.
    Indeed, as $z\in B_{\epsilon_k R_k}$, $|z_k-\frac{z}{R_k}-z_k|<\epsilon_k$. As $v_k(z_k)\in M_{inv}$, we thus obtain that for $k>k_0$, $v_k(\Omega_k)\subset M_{inv}$.

     Furthermore it follows from $R_k \text{dist}(z_k,\partial D)\to \infty$ that $\bigcup\limits_{k>k_0} \Omega_k=\mathbb{C}$. The gradient boundedness (Equation \ref{eq:boundedgradient}) translates into
    $$|\nabla \widetilde{v}_k(z)|\leq 2 \text{ on } \Omega_k.$$
    By the $C^\infty_\text{loc}$-bounds (Theorem 4.3.4 in \cite{abbashofer}), we conclude that (up to choosing a subsequence) $\widetilde{v}_k$ converges in $C^\infty_{\text{loc}}$ to a $J$-holomorphic plane
    $$\widetilde{v}=(b,v): \mathbb{C} \to \R \times M_{inv}$$
    which is non-constant because $|\nabla \widetilde{v}(0)|=1$. We compute that the energy of $\widetilde{v}$ is finite using the standard arguments. Let $Q\subset \mathbb{C}$ be a compact set and take $k>k_0$ large. We obtain
    \begin{align}
        \sup_{\phi \in \mathcal{C}} \int_Q \widetilde{v}^*_kd(\phi\alpha) \leq \sup_{\phi \in \mathcal{C}} \int_{B_{R_k}(-R_kz_k)} \widetilde{v}^*_kd(\phi \alpha) \\
        =\sup_{\phi \in \mathcal{C}}\int_D \widetilde{u}^*_k d(\phi \alpha)=E(\widetilde{u}_k)\leq C.
    \end{align}
 We now take the limit $k \to \infty$ and take the supremum over all compact $Q\subset \mathbb{C}$ to obtain that $E(\widetilde{v})<\infty$. Moreover the image of $v$ lies in a compact subset of $M_{inv}$, to be precise in $\{0\}\times \partial K$ since this is true for all the maps $v_k$.

Hence we found a finite energy plane in the $\R^+$-invariant part of $M$, and by Theorem \ref{thm:finiteenergy} this yields a periodic orbit in $M_{inv}$. By the $\R^+$-invariance, this yields a $1$-parametric family of periodic orbits in every $\{cst\}\times \partial K$.

\end{proof}

\section{The singular Weinstein conjecture}\label{sec:singularweinstein}

The last subsections lead us to revisiting the standard approach in symplectic and contact topology to apply it in the singular set-up.

	Let us consider the desingularization of almost-convex $b^{2k}$-contact forms. As a consequence of the desingularization theorem, Theorem \ref{thm:desingularizationb2k}, the properties related to the family of contact structures coming from the desingularization can be translated to properties of the initial $b^{2k}$-contact form.
	
	\begin{lemma}
		Let $(M,\alpha)$ be an almost convex $b^{2k}$-contact manifold. Consider the family of contact forms $\alpha_\epsilon$ associated to the desingularization. Assume that there exists $\epsilon$ such that there is a periodic Reeb orbit of the Reeb vector field $R_{\alpha_\epsilon}$ outside of the $\epsilon$-neighbourhood $\mathcal{N}_\epsilon$. Then this orbit corresponds to a periodic orbit of the Reeb vector field $R_\alpha$.
	\end{lemma}
	
	\begin{proof}
		The desingularization does not change the dynamics outside of the $\epsilon$-neighbourhood.
	\end{proof}
	
	Note that the same would hold for the desingularization of $b^{2k+1}$-contact structures, where the resulting geometric structure would is \emph{folded contact}, see \cite{MO}.

	Let $(M,\alpha)$ be an almost-convex compact $b^{2k}$-contact manifold of dimension $3$. Assume that the periodic Reeb orbits of $R_{\alpha_\epsilon}$ for fixed $\epsilon$ (which is known to exist due to \cite{taubes}) crosses the tubular neighbourhood $\mathcal{N}_\epsilon$ of $Z$. We will see that the desingularization changes the Reeb dynamics whenever the Reeb vector field is not everywhere regular or singular around a connected component of the critical set.
	
	\begin{lemma}
	Let $(M,\alpha)$ be a almost-convex $b^{2k}$-contact manifold. Then in the $\epsilon$-neighbourhood of the critical set, the Reeb flow associated to the desingularization is a reparametrization of the initial Reeb flow if and only if semi-locally, the Reeb vector field is everywhere regular or everywhere singular.
	\end{lemma}
	
	\begin{proof}
	As in Theorem \ref{thm:desingularizationb2k}, we write $R_\alpha= g z^{2k} \frac{\partial}{\partial z}+ X$ and the expression of the desingularized Reeb vector field is given by $R_{\alpha_\epsilon}=g \frac{1}{f_\epsilon'}\frac{\partial}{\partial z}+X$. The flow of the first one is a reparametrization of the second one if and only if $R_\alpha = fR_{\alpha_\epsilon}$ for a smooth function $f$. This is clearly only the case if the Reeb vector field is everywhere singular or everywhere regular.
	\end{proof}
	
	One is tempted to take the limit of $\epsilon \to 0$. However, the continuity of the family of periodic orbits with respect to $\epsilon$ cannot be guaranteed. Therefore, limit arguments do not work without any further assumptions on the $b^{2k}$-contact form. A necessary condition is non-degeneracy for the family of contact forms $\{\alpha_t \}_{t \in ]0,\epsilon]}$.
	
	In fact, periodic orbits can be associated to critical points of the action functional
	$$ \mathcal{A}_{\alpha_\epsilon}(\gamma)= \int_\gamma \alpha_\epsilon$$
	for $\gamma$ in the loop space $C^\infty(S^1,M)$. The non-degeneracy of the family of contact forms $\{\alpha_t \}_{t \in ]0,\epsilon]}$ can be thought of as non-degeneracy as critical points in this infinite-dimensional space.
	
	Instead of working with the desingularization (whose draw-back is the restriction on the parity of the singularity and the non-degeneracy condition), it may be more appropriate to tackle the problem using variational methods but changing the variational set-up. The authors suspect that working with the space of piece-wise smooth loops instead of $C^\infty(S^1,M)$ could be a good starting point to capture not only periodic orbits, but also heteroclinic orbits that manifest themselves by introducing the aforementioned trap construction. To seize those two different types of orbits, we introduce the notion of singular periodic orbit.
	
	\begin{definition}
		Let $M$ be a manifold with hypersurface $Z$. A \emph{singular periodic orbit} $\gamma$ is an orbit such that $\lim_{t \to  \pm \infty} \gamma(t) =p_{\pm} \in Z$ where $R_{\alpha}(p_{\pm})=0$.
	\end{definition}

	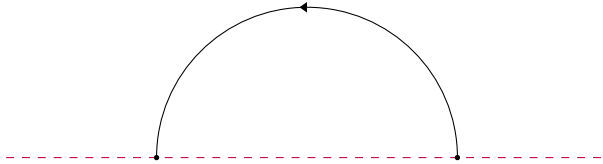
\begin{figure}[hbt!]\label{fig:singularorbit}
\begin{center}
\begin{tikzpicture}[scale=2]
 \draw (1,0) arc (0:180:1 and 1);
 \draw[dashed][color=purple] (-2,0) --  (2,0);
 \fill (1,0) circle[radius=0.5pt];
 \fill (-1,0) circle[radius=0.5pt];
\fill [shift={(-0.05,1)},scale=0.05,rotate=180]   (0,0) -- (-1,-0.7) -- (-1,0.7) -- cycle; 
\end{tikzpicture}

\caption{Singular periodic orbits}

\end{center}
\end{figure}

	Given the non-existence of periodic orbits away from the critical set as proved in Section \ref{sec:counterexample}, the question about the existence for an appropriate invariant dynamical set raises. It turns out that in the examples without periodic orbits away from $Z$, there always exists singular periodic orbits.

	\begin{example}
Consider the $b$-contact manifold $S^3\subset (\R^4,\omega)$ where $\omega \in {^b}\Omega^2(\R^4)$ is the standard $b$-symplectic form as in Example \ref{example:S3}. As already observed, the Reeb vector field given by
$$R_\alpha=2x_1^2\frac{\partial}{\partial y_1}-x_1y_1\frac{\partial}{\partial x_1}+2x_2\frac{\partial }{\partial y_2}-2y_2\frac{\partial}{\partial x_2}$$
admits infinitely many periodic orbits on $Z=\{x_1=0\}$ but none away from the critical set. However, there is a singular periodic orbit originating from the two fixed points at the critical surface given by $(0,\pm 1,0,0)$.
The orbit is topologically given by circle and dynamically speaking, it has two marked fixed points corresponding to the two fixed points.
\end{example}

\begin{example}
Consider the example of the $3$-torus $(\mathbb{T}^3,\sin \varphi \frac{dx}{\sin x}+\cos \varphi dy)$ as in Example \ref{example:3torus}. The Reeb vector field is given by
$$R_\alpha=\sin \varphi \sin x \frac{\partial}{\partial x}+\cos \varphi \frac{\partial}{\partial y}$$
has no periodic orbits away from $Z$ and is singular on the critical set when $\varphi=\pm \frac{\pi}{2}$. The orbits outside of $Z$ limiting to the fixed points set are all singular orbits.
\end{example}

	 An appropriate reformulation of Weinstein in the singular setting would be, that there either exists periodic orbits away from the critical set (as is proved under the assumptions as in Section \ref{sec:bmhofer}) or a \emph{singular periodic orbit}.

	\begin{conjecture}[Singular Weinstein conjecture]
		Let $(M,\alpha)$ be a compact  $b^m$-contact manifold. Then there exists at least one singular periodic orbit.
	\end{conjecture}

    Note that the possible existence of $b^m$-plugs does not contradict this conjecture: morally, the $b^m$-contact plugs changes periodic orbits to singular periodic orbits. Plugs for the $b^m$-Reeb flow would give rise to $b^m$-symplectic manifolds with proper smooth Hamiltonian functions that do not admit periodic orbits on any level-set. Once more, our construction replaces periodic orbits by singular periodic orbits.

In particular if the singular Weinstein conjecture holds true in this new singular set-up then,

\begin{corollary}[Corollary of singular Weinstein conjecture]
Any Beltrami field in a manifold with a cylindrical end has at least one of the two:
\begin{enumerate}
\item a periodic orbit.
\item an orbit that goes to infinity for $t \rightarrow +\infty$ and $t \rightarrow -\infty$.
\end{enumerate}
\end{corollary}

Both situations are illustrated below in Figure \ref{fig:bbeltrami}

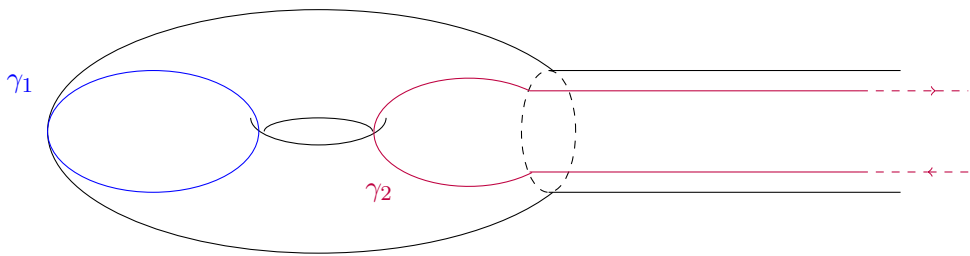
\begin{figure}[!h]
\begin{center}
\begin{tikzpicture}[scale=1.8]
%

\draw (-2,0,0) arc (-180:-30:2 and 0.9);
\draw (-2,0,0) arc (-180:-330:2 and 0.9);

\draw (0.4,0,0) arc (0:180:0.4 and 0.1);
\draw (-0.5,0.1,0) arc (-180:0:0.5 and 0.2);

\draw (1.7,0,0)[dashed] ellipse (0.2 and 0.45);

\draw (1.7,0.45,0)--(4.3,0.45,0);
\draw (1.7,-0.45,0)--(4.3,-0.45,0);

\draw [blue] (-1.22,0,0) ellipse (0.78 and 0.45);

\draw [purple] (1.56,0.3,0)--(4,0.3,0);
\draw [purple] (1.56,-0.3,0)--(4,-0.3,0);
\draw [purple] (1.56,0.3,0) arc (50:312:0.7 and 0.4);
\draw [purple,dashed, ->] (4,0.3,0)--(4.56,0.3,0);
\draw [purple,dashed, ->] (4.8,-0.3,0)--(4.5,-0.3,0);
\draw [purple,dashed] (4.56,0.3,0)--(4.8,0.3,0);
\draw [purple,dashed] (4,-0.3,0)--(4.5,-0.3,0);

\draw (0.45,-0.45,0)[purple] node[scale=1]{$\gamma_2$};
\draw (-2.2,0.35,0)[blue] node[scale=1]{$\gamma_1$};
\end{tikzpicture}
\end{center}
 \caption{Periodic orbits on the regular part (in blue) and periodic orbits going to infinity (in purple).}
 \label{fig:bbeltrami}
 \end{figure}

 The authors believe that techniques, similar to \cite{Ginzburg,GinzburgGurel} can be adapted to give examples of level-sets of $b^m$-symplectic manifolds containing no singular periodic orbit. This would be a counter-example to the singular Hamiltonian Seifert conjecture.

	\begin{conjecture}
		Let $(M,\omega)$ be a $b^m$-symplectic manifold. There exists a $H\in C^\infty(M)$ proper, smooth Hamiltonian whose level-sets do not contain any singular periodic orbits.
	\end{conjecture}
	
The definition of singular periodic orbits opens the door to an exciting brave new world which we are eager to explore: Can these solutions be accepted as "periodic orbits" in the singular context and thus can Weinstein's conjecture  be reformulated in those terms? Can the Rabinowitz machinery in \cite{Rabi} be extended to this set-up? Can a Floer complex be built upon the critical points corresponding to  periodic orbits with marked singular points? Understanding these questions is an endeavour that leads to a paramount extension of the Floer techniques to the barely unexplored land of singular symplectic and contact topology and, in particular, ventures into Poisson topology.

\end{document}